\documentclass[a4paper,11pt,draft]{amsart}

\usepackage[T1]{fontenc}

\usepackage[utf8]{inputenc}

\usepackage{geometry}

\usepackage{lmodern}

\usepackage[final]{hyperref}

\usepackage{enumerate}

\usepackage{todonotes}

\usepackage{color}

\usepackage[abbrev]{amsrefs}

\title%
  [Odd symmetry of nodal solutions of Choquard]%
  {Odd symmetry of least energy nodal solutions for the Choquard equation}

\author{David Ruiz}
  \address{%
    Universidad de Granada\\
    Departamento de An\'alisis Matem\'atico\\
    Campus Fuentenueva\\
    18071 Granada\\
    Spain}
  \email{daruiz@ugr.es}

\author{Jean Van Schaftingen}
  \address{%
    Universit\'e Catholique de Louvain\\
    Institut de Recherche en Math\'ematique et Phy\-sique\\
    Chemin du Cyclotron 2 bte L7.01.01\\
    1348 Louvain-la-Neuve \\
    Belgium}
  \email{Jean.VanSchaftingen@UCLouvain.be}

\thanks{D.\thinspace R. was supported by the FEDER-MINECO Grant MTM2015-68210-P and by J.
Andalucia (FQM-116).
J.\thinspace V.\thinspace S. was supported by the Projet de Recherche (Fonds de la Recherche Scientifique--FNRS) T.1110.14 ``Existence and asymptotic behavior of solutions to systems of semilinear elliptic partial differential equations''. }

\newcommand{\Rset}{\mathbb{R}}
\newcommand{\R}{\mathbb{R}}
\newcommand{\Nset}{\mathbb{N}}

\newcommand{\RN}{\mathbb{R}^N}
\newcommand{\abs}[1]{\lvert #1 \rvert}
\newcommand{\bigabs}[1]{\bigl\lvert #1 \bigr\rvert}
\newcommand{\Bigabs}[1]{\Bigl\lvert #1 \Bigr\rvert}
\newcommand{\norm}[1]{\lVert #1 \rVert}
\newcommand{\dualprod}[2]{\langle #1, #2 \rangle}
\newcommand{\scalprod}[2]{({#1} \vert {#2})}
\newcommand{\st}{\;\vert\;}
\newcommand{\dif}{\,\mathrm{d}}

\newcommand{\gst}{\mathrm{gst}}
\newcommand{\nod}{\mathrm{nod}}

\theoremstyle{plain}
\newtheorem{proposition}{Proposition}[section]
\newtheorem{lemma}[proposition]{Lemma}

\newtheorem{theorem}{Theorem}

\theoremstyle{definition}
\newtheorem{remark}{Remark}[section]
\numberwithin{equation}{section}

\theoremstyle{plain}
\newtheorem{claim}{Claim}
\newcommand{\resetclaim}{\setcounter{claim}{0}}
\newenvironment{proofclaim}
  [1]
  [Proof of the claim]
  {\begin{proof}[#1]}
  {\end{proof}}

\theoremstyle{plain}
\newtheorem{step}{Step}
\newcommand{\resetstep}{\setcounter{step}{0}}

\begin{document}
\begin{abstract}
We consider the Choquard equation (also known as
stationary Hartree equation or Schr\"odinger--Newton equation)
\[
 -\Delta u + u = (I_\alpha \ast \abs{u}^p) \abs{u}^{p - 2}u.
\]
Here $I_\alpha$ stands for the Riesz potential of order $\alpha \in (0,N)$, and \(\frac{N - 2}{N + \alpha} < \frac{1}{p} \le \frac{1}{2}\). We prove that least energy nodal solutions have an odd symmetry with respect to a hyperplane when \(\alpha \) is either close to \(0\) or close
to \(N\).
\end{abstract}

\maketitle

\section{Introduction}

In this paper we are interested in the Choquard equation (also known as
stationary Hartree equation or Schr\"odinger--Newton equation)
\begin{equation}
\label{eqChoquard}
% \label{theproblem}
 -\Delta u + u = (I_\alpha \ast \abs{u}^p) \abs{u}^{p - 2}u,
\end{equation}
where \(I_\alpha : \Rset^N \to \Rset\) denotes the Riesz potential,
which is defined for each \(x \in \Rset^N \setminus \{0\}\) by
\begin{align}
\label{eqRiesz}
  I_\alpha (x) = \frac{A_\alpha}{\abs{x}^{N - \alpha}}, \
  \text{ where } \
  A_\alpha = \frac{\Gamma(\tfrac{N-\alpha}{2})}
		    {\Gamma(\tfrac{\alpha}{2})\pi^{N/2}2^{\alpha} }.
\end{align}
Problem \eqref{eqChoquard} is the Euler--Lagrange equation of the \emph{Choquard action functional} \(J_\alpha : H^1(\Rset^N) \to \Rset\) which is defined for each \(u\) in the Sobolev space \(H^1 (\Rset^N)\) by
\begin{equation}%
  \label{thefunctional}%
  J_{\alpha}(u)
  = \frac{1}{2} \int_{\Rset^N}  \big ( \abs{\nabla u}^2 + \abs{u}^2 \big )
  - \frac{1}{2p} \int_{\Rset^N} \bigl(I_{\alpha} \ast \abs{u}^p \bigr) \abs{u}^{p}.
\end{equation}

When \(N = 3\), \(\alpha = 2\) and \(p = 2\), the equation \eqref{eqChoquard} arises in Pekar's model of the polaron \citelist{\cite{Pekar1954}\cite{Lieb1977}}. It has also appeared by introducing classical Newtonian gravitation in quantum physics \citelist{\cite{Diosi1984}\cite{Jones1995}\cite{Penrose1996}}.
The Choquard equation has been the object of many mathematical works (see \cite{MVSReview}).

The existence of groundstate solutions (or least energy solutions) is quite well-known, see \citelist{\cite{Lieb1977}\cite{Lions1980}\cite{Menzala1980}\cite{Tod-Moroz-1999}\cite{Moroz-Penrose-Tod-1998}\cite{MVSGNLSNE}}. Those solutions are positive and radially symmetric. The uniqueness is known in some cases (see for instance \cite{Lieb1977}).
It is also well-known that problem \eqref{eqChoquard} admits sign-changing solutions with various symmetries
\citelist{\cite{CingolaniClappSecchi2012}\cite{CingolaniSecchi}\cite{ClappSalazar}\cite{Weth2001}*{Theorem 9.5}}.

Recently, other type of sign-changing solutions have been found for the Choquard equation. If \(\frac{N - 2}{N + \alpha} < \frac{1}{p} < \frac{N}{N + \alpha}\) there exist solutions which are odd with respect to a hyperplane of \(\Rset^N\), see \cite{GhimentiVanSchaftingen}. Those solutions have minimal energy among all solutions with that symmetry. Furthermore, there are also nodal solutions which minimize the energy in the so-called Nehari nodal set in the case \(\frac{N - 2}{N + \alpha} < \frac{1}{p} \le \frac{1}{2}\) (see \citelist{\cite{GhimentiVanSchaftingen}\cite{GhimentiMorozVanSchaftingen}}), which will be called \emph{least energy nodal solutions}. We point out that, in both cases, those solutions do not have a counterpart in the framework of the usual stationary nonlinear Schr\"{o}dinger equation.

At this point, it is quite reasonable to ask whether those solutions coincide; in other words, \emph{whether the least energy nodal solutions are odd-symmetric with respect to a hyperplane}. The aim of this work is to give an affirmative answer to that question, if the order \(\alpha\) of the Riesz potential is either close to \(0\) or close to \(N\).

We first state the result for \(\alpha\) close to \(0\).

\begin{theorem}
\label{theoremMain0}
If \( 1 - \frac{2}{N} < \frac{1}{p} \leq \frac{1}{2}\),
then there exists \(\alpha_\star \in (0, N)\) such that for any $\alpha \in (0,\alpha_\star)$,
any least energy nodal solution \(u_{\alpha} \in H^1 (\Rset^N)\) of the Choquard equation \eqref{eqChoquard} is odd with respect to a hyperplane of \(\Rset^N\).
\end{theorem}

By odd, we mean that there exists a reflection \(R : \Rset^N \to \Rset^N\) of the Euclidean space \(\Rset^N\) with respect to an affine hyperplane of \(\Rset^N\) such that \(u \circ R = - u\) in \(\Rset^N\).
%This suggests that least energy  odd solutions might minimize the action among sign-changing solutions. (The condition \(p \ge 2\) was shown to be necessary for the existence of minimizers on the Nehari nodal set \cite{GhimentiVanSchaftingen}, this does not necessarily forbid the existence least energy  nodal solutions.)

\medskip

For the case where $\alpha$ is close to $N$ our result is the following.

\begin{theorem}%
\label{theoremMainN}
If \(\frac{1}{2} - \frac{1}{N} < \frac{1}{p} < \frac{1}{2}\),
then there exists  $\alpha^\star \in (0, \ N)$
such that for any \(\alpha \in (\alpha^\star, N)\),
any least energy nodal solution $u_{\alpha} \in H^1 (\Rset^N)$ of the Choquard equation \eqref{eqChoquard} is odd with respect to a hyperplane.
\end{theorem}

In general, the hypothesis $\frac{N - 2}{N + \alpha} < \frac{1}{p} < \frac{N}{N + \alpha}$ in necessary by a Pohozhaev-type inequality for the existence of sufficiently regular finite energy solutions to \eqref{eqChoquard} \cite{MVSGNLSNE}*{Theorem~2}. Since Theorems~\ref{theoremMain0} and \ref{theoremMainN} are concerned with the case $\alpha \to 0$ and $\alpha \to N$ respectively, the restrictions on $p$ imposed are the natural limit of these conditions.
% Regarding the possibility $p=2$, this is not included in Theorem \ref{theoremMainN} because of the limit problem that appears in the proofs.

The proofs use an argument by contradiction. We study the behavior of least energy nodal solutions $u_\alpha$ of the Choquard equation \eqref{eqChoquard} when either \(\alpha \to 0\) or \(\alpha \to N\). This process leads us naturally to certain limit problems. If $\alpha \to 0$ the limit problem is just a usual stationary nonlinear Schr\"{o}dinger equation, but in the case $\alpha \to N$ the equation includes an additional coefficient depending on the nonlocal quantity
$\Vert u\Vert_{L^p}$.

A crucial ingredient of the proofs is the asymptotics of the Riesz potential energy $ \int_{\R^N} (I_{\alpha} \ast \abs{u}^p) \abs{u}^p$. In the r\'egime \(\alpha \to 0\), the approximation is uniform on bounded sets of the Sobolev space \(H^1 (\Rset^N)\) (see \S~\ref{sectionConcentratingRiesz} below), which suits perfectly in our proofs. When \(\alpha \to N\), the analysis is more delicate, because there is only a unilateral uniform approximation property on bounded sets (see \S~\ref{sectionDelocalizingRiesz}).

We point out that the family $u_\alpha$ \emph{does not converge} to a solution to the limit problem, even up to translations in \(\Rset^N\) and up to the extraction of a subsequence, an issue that makes our proof more involved.
In the proof of Theorem~\ref{theoremMain0} we show that the sequence of solutions actually forms a \emph{Palais--Smale sequence} for the limit equation. As a consequence, our solutions behave asymptotically like differences of two positive solutions of the local problem moving away one from the other.  With this in hand, we use the nondegeneracy of solutions to the local problem to conclude that the solution has an odd symmetry.

For Theorem~\ref{theoremMainN} we need to describe more accurately the solutions, and we prove that the positive and negative parts of $u_\alpha$ converge to a groundstate of the corresponding limit problem.
Moreover we also need to estimate the distance between the two bumps: it is going to infinity but slowly enough to preserve the interaction between the bumps as much as possible.
In contrast with Theorem~\ref{theoremMain0} which still holds for \emph{low-energy nodal solutions} (see Proposition~\ref{puf} below), the proof of Theorem~\ref{theoremMainN} uses essentially the minimizing character of the nodal solutions.

%Because the distance between the bumps is not determined univocally, our proof does not show however the uniqueness up to isometries of the least action odd solutions.

The rest of the paper is organized as follows. In Section~\ref{sectionPreliminaries} we state some known results about groundstate solutions and nodal solutions of the Choquard equation.
We also review properties of the limit problems that we encounter in the proofs.
Section~\ref{sectionAsymptoticRiesz} is devoted to the study of the asymptotic behavior of the Riesz potential when $\alpha$ tends to $0$ or $N$.
Theorems~\ref{theoremMain0} and~\ref{theoremMainN} are proved in Sections~\ref{sectionProofTheoremMain0} and~\ref{sectionProofTheoremMainN}, respectively.

\section{Preliminaries}
\label{sectionPreliminaries}

\subsection{Groundstates and least energy nodal solutions of the Choquard equations}

%\subsection{Groundstates of the Choquard equation}

\medskip

For any \(\alpha \in (0, N)\) and \(p \in (1, \infty)\) such that
\(\frac{N - 2}{N + \alpha}< \frac{1}{p} <\frac{N}{N + \alpha}\), the solutions in \(H^1 (\Rset^N)\) to the Choquard equation \eqref{eqChoquard} correspond to critical points of the energy functional \(J_\alpha\) defined on \(H^1 (\Rset^N)\) by \eqref{thefunctional}.
The Choquard equation \eqref{eqChoquard} has a positive, radially symmetric groundstate solution \(U_\alpha \in H^1 (\Rset^N)\) \citelist{\cite{Lieb1977}\cite{Lions1980}\cite{Menzala1980}\cite{Tod-Moroz-1999}\cite{Moroz-Penrose-Tod-1998}\cite{MVSGNLSNE}} whose energy level will be denoted by $c_\alpha^\gst = J_\alpha(U_\alpha)$.

The groundstate level \(c_\alpha^\gst\) has many different characterizations \cite{MVSGNLSNE}*{\S 2.1};
it can be obtained as a minimum
\[
 c_\alpha^\gst = \inf\,\bigl\{ J_{\alpha}( u) \st u \in \mathcal{N}_\alpha \bigr\},
\]
on the \emph{Nehari manifold} which is defined by
\[
  \mathcal{N}_\alpha
  = \bigl\{ u \in H^1(\R^N) \setminus \{ 0\}
    \st \dualprod{J_{\alpha}'(u)}{u}=0  \bigr\}.
\]
The level \(c_\alpha^\gst\) can be equivalently characterized variationally as a minimax level:
\begin{equation}
\label{eqChoquardGroundstateMinimax}
  c_\alpha^\gst = \inf\,\Bigl\{\max_{t \ge 0} J_{\alpha}(t u) \st u \in H^1(\Rset^N) \setminus \{0\}\Bigr\}.
\end{equation}

%\subsection{Minimal action nodal solutions}

\medskip

We now turn our attention to \emph{nodal solutions for the Choquard equation}. As for local problems in bounded domains (see
\citelist{%
  \cite{CeramiSoliminiStruwe1986}%
  \cite{CastroCossioNeuberger1997}%
  \cite{CastorCossioNeuberger1998}%
}) least energy  nodal solutions can be constructed when \(p \ge 2\)  by minimizing the action functional on the Nehari nodal set %
\citelist{\cite{GhimentiMorozVanSchaftingen}%
\cite{GhimentiVanSchaftingen}}:
\[
  c_{\alpha}^{\nod} = \inf\,\bigl\{ J_{\alpha}(u) \st u \in \mathcal{N}_\alpha^{\nod}\bigr\},
\]
where  the \emph{Nehari nodal set} \(\mathcal{N}_\alpha^{nod}\) is defined by
\[
  \mathcal{N}_\alpha^{\nod}
  = \bigl\{ u \in H^1(\R^N)
    \st \dualprod{J_{\alpha}'(u)}{u^+}=0, \ \dualprod{J_{\alpha}'(u)}{u^-}=0,\ u^+\neq 0\text{ and } u^- \neq 0\bigr\}.
\]
This level can also be characterized by
\begin{equation}
  c_{\alpha}^{\nod}
  = \inf\,\Bigl\{\max_{t, s \ge 0} J_{\alpha}(tu^+ + s u^-) \st u\in H^1(\Rset^N), \ u^+ \neq 0 \text{ and } u^- \neq 0\Bigr\}.
\end{equation}
This can be seen as follows (see \cite{GhimentiVanSchaftingen}): if \(u \in H^1 (\Rset^N)\), and if \(u^+ \ne 0\) and \(u^- \ne 0\),
then for every \(\sigma, \tau \in [0, \infty)\), we have
\begin{multline*}
\label{expansion}
  J_{\alpha}(\tau^{1/p} u^+ + \sigma^{1/p} u^-) \\
  = \frac{\tau^{2/p}}{2} \int_{\Rset^N} \abs{\nabla u^+}^2 + \abs{u^+}^2
    + \frac{\sigma^{2/p}}{2} \int_{\Rset^N} \abs{\nabla u^-}^2 + \abs{u^-}^2 \\
  - \frac{1}{2p} \int_{\Rset^N} \bigl(I_\alpha \ast (\tau \abs{u^+}^p + \sigma \abs{u^-}^p) \bigr) \bigl(\tau \abs{u^+}^p + \sigma \abs{u^-}^p\bigr);
\end{multline*}
the right-hand side is a strictly concave function in the variables $\sigma$, $\tau$ (see \cite{LiebLoss2001}*{Theorem 9.8}) and achieves its maximum at a unique point \((\sigma, \tau) \in (0, \infty)^2\). If \(u \in \mathcal{N}_\alpha^\nod\), then \((\tau, \sigma) = (1, 1)\) is a critical point and the conclusion thus follows.

The level \(c_{\alpha}^\nod\) can be estimated by the groundstate level \(c_\alpha^\gst\)
\cite{GhimentiVanSchaftingen}.

%TODO: Precise Ghimenti & Van Schaftingen

\begin{proposition}
\label{propositionControlGroundNodal}
If \(p \ge 2\) and \(((N - 2)p - N)_+ < \alpha < N\), then
\[
  c_{\alpha}^{\nod} < 2 c_\alpha^\gst.
\]
\end{proposition}

\subsection{Limiting problems}

When \(\alpha \to 0\), the Choquard equation \eqref{eqChoquard} reduces at least formally to
the nonlinear Schr\"odinger equation with exponent \(q = 2 p\)
\begin{equation} %
  \label{limit0}%
  - \Delta u + u = \abs{u}^{q-2}u.
\end{equation}
The latter equation \eqref{limit0} is the Euler--Lagrange equation of the energy functional $\Phi_q: H^1(\R^N) \to \R$ defined for each \(u \in H^1 (\Rset^N)\) by
\begin{equation}
  \label{functional0}%
  \Phi_q (u)
  = \frac{1}{2} \int_{\Rset^N} \bigl( \abs{\nabla u}^2 + \abs{u}^2 \bigr) - \frac{1}{q} \int_{\Rset^N} \abs{u}^{q}.
\end{equation}
Problem \eqref{limit0} has a positive groundstate \(U\) which is \emph{radially symmetric}, \emph{unique up to translations} and \emph{nondegenerate}, that is, any solution $v \in H^1(\R^N)$ to the linearized problem
\[
  - \Delta v + v = (q-1)U^{q-2}v
\]
is a directional derivative of the function \(U\): it can be written \(v = h \cdot \nabla u\), for some constant vector \(h \in \Rset^N\) \citelist{\cite{Kwong1989}\cite{Weinstein1985}\cite{Oh1990}}.

The \emph{groundstate level} $\gamma_{q} = \Phi_q (U)$ has many different variational characterizations.
We will be using the fact that the groundstate solution minimizes
\begin{equation}
  \label{eqSchrodingerGroundstateMinimax}
    \gamma_{q} = \inf\, \Bigl\{ \max_{t \ge 0} \Phi_q (t u) \st u \in H^1(\Rset^N)\setminus \{0\}\Bigr\}.
\end{equation}
% \begin{equation}
%   \label{supremum} \gamma= \sup \left \{ \int_{\Rset^N} \abs{u}^{q}\st \ u \in S \right \}, \ S= \left \{ u \in H^1(\R^N)\st \norm{u}_{H^1 (\Rset^N)}=1 \right\}.
% \end{equation}
Indeed, it can be proved that the above infimum is attained by the groundstate $U$ and $\max_{t \ge 0} \Phi_q (t U)= \Phi_q(U)$.

It is also well known that any other solution $u$ of \eqref{limit0} must change sign and satisfies
\begin{equation}
\label{strictOtherSolutions0}
 \Phi_q (u) > 2 \gamma_q.
\end{equation}

Finally, the
behavior of the Palais--Smale sequences of \eqref{limit0} has been fully described \citelist{\cite{bahri-lions}*{Proposition II.1}\cite{Willem1996}*{Theorem~8.4}\cite{BenciCerami1987}}:

\begin{lemma} \label{PS} If the sequence \((u_n)_{n \in \Nset}\) in \(H^1 (\Rset^N)\) is a Palais--Smale sequence for the functional \(\Phi_q\), that is, if the sequence \((\Phi_q(u_n))_{n \in \Nset}\) is bounded in \(\Rset\) and if the sequence \((\Phi_q'(u_n))_{n \in \Nset}\) converges to \(0\) in \( H^{-1} (\Rset^N)\), then, there exists an integer $m \geq 0$, sequences \((a_n^i)_{n \in \Nset}\) in \(\Rset^N\) for $i= 1, \dotsc, m$), and nonzero solutions $U_i$ of \eqref{limit0} such that, as \(n \to \infty\),
\begin{enumerate}[i)]
\item $ u_n -  \sum_{i =1}^m U_i(\cdot- a_n^i) \to 0$ strongly in \(H^1 (\Rset^N)\),
\item $\Phi_q(u_n) \to \sum_{i =1}^m \Phi_q (U_i)$,
\item $\abs{a_n^i - a_n^j} \to +\infty$ if $i \neq j$.
\end{enumerate}
\end{lemma}

\medskip

In the study of the Choquard equation \eqref{eqChoquard} for
\(\alpha\) close to \(N\), we will encounter the following variant of the nonlinear Schr\"odinger equation:
\begin{equation}
  \label{limit1}
  -\Delta u + u
  = \mu \Bigl(\int_{\Rset^N} \abs{u}^p \Bigr) \abs{u}^{p-2}u,
\end{equation}
for some parameter \(\mu > 0\).
This equation \eqref{limit1} is the Euler--Lagrange equation of the energy functional $\Psi_{p,\mu}: H^1(\R^N) \to \R$ defined for each \(u \in H^1 (\Rset^N)\) by
\begin{equation}
  \label{functional1}%
  \Psi_{p, \mu}(u)
  = \frac{1}{2} \int_{\Rset^N} \bigl(\abs{\nabla u}^2 + \abs{u}^2\bigr)
    - \frac{\mu}{2p} \Bigl(\int_{\Rset^N}\abs{u}^{p}\Bigr)^2.
\end{equation}

The solutions of the problems \eqref{limit0} and \eqref{limit1} are related to each other.
Indeed, if \(u \in H^1 (\Rset^N)\) is a solution of the equation \eqref{limit0} with $q=p$, we define
\[
 v = \frac{u}{\Bigl(\displaystyle \mu \int_{\Rset^N} \abs{u}^p \Bigr)^\frac{1}{2p - 2}}.
\]
We observe that \(v\) is solution to problem \(\eqref{limit1}\) and that
\begin{equation}%
 \label{energies}%
 \Psi_{p, \mu} (v)
 = \frac{\frac{1}{2} - \frac{1}{2p}}
        {\mu^{\frac{1}{p-1}}}
        \biggl(\frac{\Phi_p (u)}{\frac{1}{2} - \frac{1}{p}}\biggr)^{\frac{p-2}{p-1}}.
\end{equation}
Conversely, if \(v \in H^1 (\Rset^N)\) is a solution of problem \eqref{limit1}, then the function
\[
 u = \biggl(\mu \int_{\Rset^N} \abs{v}^p \biggr)^\frac{1}{p - 2} v
\]
is a solution of equation \eqref{limit0} with $q=p$.

The groundstate  \(V\) of problem \eqref{limit1} inherits the sign, uniqueness and symmetry properties
of the groundstate \(U\) of problem \eqref{limit0}.
The groundstate levels are related as follows:
\begin{equation}
\label{eqKappa}
  \kappa_{p, \mu} = \Psi_{p, \mu} (V) =  \frac{\frac{1}{2} - \frac{1}{2p}}
        {\mu^{\frac{1}{p-1}}}
        \biggl(\frac{\Phi_p (U)}{\frac{1}{2} - \frac{1}{p}}\biggr)^{\frac{p-2}{p-1}}
        = \frac{\frac{1}{2} - \frac{1}{2p}}
        {\mu^{\frac{1}{p-1}}}
        \biggl(\frac{\gamma_p}{\frac{1}{2} - \frac{1}{p}}\biggr)^{\frac{p-2}{p-1}}.
\end{equation}
The groundstate level $\kappa_{p,\mu}$ can be characterized variationally as
\begin{equation}
\label{eqCharactKappa}
 \kappa_{p, \mu} = \inf\,\Bigl\{\max_{t \ge 0} \Psi_{p, \mu} (t u) \st u \in H^1(\Rset^N)\setminus\{0\}\Bigr\} = \inf \,\bigl\{ \Psi_{p,\mu}(u): \ u \in \mathcal{N}_{p,\mu}\bigr\},
\end{equation}
where the Nehari manifold associated to \eqref{limit1} is defined by
\begin{equation}%
\label{othernehari}%
  \mathcal{N}_{p,\mu} = \bigl\{ u \in H^1(\Rset^N)\setminus\{0\},\ \Psi_{p,\mu}'(u)(u)=0\bigr\}.
\end{equation}

The following lemma will be needed later in the proofs, and basically states that minimizing sequences in $\mathcal{N}_{p, \mu}$ are convergent to the groundstate, up to translations. Its proof is standard and will be omitted.

\begin{lemma} \label{salut}
Let \((u_k)_{k \in \Nset}\) be a sequence in \(\mathcal{N}_{p,\mu}\).  If $\Psi_{p,\mu}(u_k) \to \kappa_{p, \mu}$ as \(k \to \infty\), then there exists sequences \((\xi_k)_{k \in \Nset}\) in \(\R^N\) and \((\gamma_k)_{k \in \Nset}\) in \(\{-1, 1\} \) such that
\[
  u_k  - \gamma_k V(\cdot - \xi_k) \to 0 \text{ in \(H^1(\Rset^N)\)},
\]
where $V$ is the groundstate of problem \eqref{limit1}.
\end{lemma}

Finally, if \(u\) is a sign-changing solution of problem \eqref{limit1}, then, in view of \eqref{strictOtherSolutions0}
\begin{equation}
  \label{strictOtherSolutionsN}%
  \Psi_{p, \mu} (u) > 2^{\frac{p-2}{p-1}} \kappa_{p, \mu}
  = 2 \kappa_{p, 2 \mu}.
\end{equation}

\bigskip

%To finish this account of preliminary results, we state the following lemma, which will be needed in the course of our proofs.
%
%\begin{lemma}[\citelist{\cite{Bogachev2007}*{proposition~4.7.12}\cite{Willem2013}*{proposition 5.4.7}}]
%\label{lemmaLocalWeak}
%If \((f_n)_{n \in \Nset}\) converges strongly to \(f\) in \(L^q (\Rset^N)\),
%if \((g_n)_{n \in \Nset}\) converges strongly to \(g\) in \(L^\frac{q}{q - 1}_{\mathrm{loc}} (\Rset^N)\) and if \((g_n)_{n \in \Nset}\) is bounded in \(L^\frac{q}{q - 1} (\Rset^N)\), then
%\[
%  \lim_{n \to \infty} \int_{\Rset^N} f_n g_n = \int_{\Rset^N} fg.
%\]
%\end{lemma}
%
%\todo[inline]{David: Jean, this lemma should be changed in my opinion. With the modifications made in the rest of the paper, what we need now is the following:}
%
%\begin{lemma}
%\label{lemmaLocalWeak}
%If \(U\) is a $L^\infty$ function with $\lim_{|x|\to +\infty} U(x)=0$, \((g_n)_{n \in \Nset}\) is a sequence converging strongly to \(g\) in \(L^q_{\mathrm{loc}} (\Rset^N)\) with \(\|g_n\|_{L^q(\Rset^N)}\) bounded, then
%\[
%  U g_n \to Ug \mbox{ in }L^q \mbox{ sense.}
%\]
%\end{lemma}
%
%\todo[inline]{David: I don't know if it worths the name of lemma, or it is quite standard, what do you think?}

\section{Asymptotic behavior of the Riesz potential energy}
\label{sectionAsymptoticRiesz}

\subsection{Concentrating Riesz potentials}
\label{sectionConcentratingRiesz}

\medskip

In order to understand the asymptotic behavior of the Riesz potential energy as \(\alpha \to 0\), we rely on the following
\(L^2\) estimate.

\begin{lemma}
\label{lemmaConvolutions}
Let \(s \in (0, N)\) and \(\beta \in (0, \infty)\).
For every \(f, g \in L^2 (\Rset^N)\) and every \(\alpha \in (0, \beta]\) such that
\(I_{\beta} \ast f \in L^2 (\Rset^N)\) and \((-\Delta)^{s/2} f \in L^2 (\Rset^N)\), one has
\[
 \Bigabs{\int_{\Rset^N} (I_\alpha \ast f) g  - \int_{\Rset^N} fg}\le
 \bigl(\tfrac{\alpha}{\beta} \norm{I_{\beta} \ast f}_{L^2 (\Rset^N)} + \tfrac{\alpha}{s} \norm{(-\Delta)^\frac{s}{2} f}_{L^2 (\Rset^N)}\bigr)\norm{g}_{L^2 (\Rset^N)}.
\]
\end{lemma}
\begin{proof}
If \(\widehat{f}\) and \(\widehat{g}\) denote the Fourier transforms of the functions \(f\) and \(g\), we have by the Plancherel theorem and by the formula for the Fourier transform of a Riesz potential
\[
  \int_{\Rset^N} (I_\alpha \ast f) g  - \int_{\Rset^N} fg
  = \int_{\Rset^N} \bigl((2 \pi \abs{\xi})^{-\alpha} - 1)\, \widehat{f}(\xi)\, \overline{\widehat{g}(\xi)}\dif \xi\,.
\]
We first observe by the Young inequality that if \(2 \pi \abs{\xi} \le 1\), then
\[
 1 \le (2 \pi \abs{\xi})^{-\alpha}
 \le 1 - \tfrac{\alpha}{\beta} + \tfrac{\alpha}{\beta}(2 \pi \abs{\xi})^{-\beta}
 \le 1 + \tfrac{\alpha}{\beta}(2 \pi \abs{\xi})^{-\beta}
\]
and therefore
\[
 \bigabs{(2 \pi \abs{\xi})^{-\alpha} - 1} = (2 \pi \abs{\xi})^{-\alpha} - 1  \le \tfrac{\alpha}{\beta} (2 \pi \abs{\xi})^{-\beta}.
\]
It follows thus by the Cauchy--Schwarz inequality that
\[
\begin{split}
 \Bigabs{\int_{B_{1/(2 \pi)}} \bigl((2 \pi \abs{\xi})^{-\alpha} - 1)\, \widehat{f}(\xi) \overline{\widehat{g}(\xi)} \dif \xi}
 &\le \tfrac{\alpha}{\beta} \Bigabs{\int_{B_{1/(2 \pi)}} (2 \pi \abs{\xi})^{-\beta}\, \widehat{f}(\xi) \overline{\widehat{g}(\xi)} \dif \xi}\\
 &\le \tfrac{\alpha}{\beta} \Bigl(\int_{\Rset^N} (2 \pi \abs{\xi})^{-2\beta} \bigabs{\widehat{f}(\xi)}^2\dif \xi \Bigr)^\frac{1}{2}
 \Bigl(\int_{\Rset^N} \bigabs{\widehat{g}(\xi)}^2\dif \xi \Bigr)^\frac{1}{2}\\
 &= \tfrac{\alpha}{\beta} \norm{I_{\beta} \ast f}_{L^2 (\Rset^N)} \norm{g}_{L^2 (\Rset^N)}.
\end{split}
\]
On the other hand, if \(2 \pi \abs{\xi} \ge 1\), we have, by Young's inequality again,
\[
 (2 \pi \abs{\xi})^{-\alpha}
 \le 1 \le \tfrac{s}{\alpha + s}(2 \pi \abs{\xi})^{-\alpha}
 + \tfrac{\alpha}{\alpha + s}(2 \pi \abs{\xi})^{s}
 \le (2 \pi \abs{\xi})^{-\alpha}
 + \tfrac{\alpha}{s}(2 \pi \abs{\xi})^{s},
\]
so that
\[
 \bigabs{(2 \pi \abs{\xi})^{-\alpha} - 1} = 1 - (2 \pi \abs{\xi})^{-\alpha}\le \tfrac{\alpha}{s} (2 \pi \abs{\xi})^{s}.
\]
Therefore,
\[
\begin{split}
 \Bigabs{\int_{\Rset^N \setminus B_{1/(2 \pi)}} \bigl((2 \pi \abs{\xi})^{-\alpha} - 1)\, \widehat{f}(\xi) \overline{\widehat{g}(\xi)} \dif \xi}
 &\le \tfrac{\alpha}{s} \Bigabs{\int_{\Rset^N \setminus B_{1/(2 \pi)}} (2 \pi \abs{\xi})^{-s}\, \widehat{f}(\xi) \overline{\widehat{g}(\xi)} \dif \xi}\\
 &\le \tfrac{\alpha}{s} \Bigl(\int_{\Rset^N} (2 \pi \abs{\xi})^{2s} \bigabs{\widehat{f}(\xi)}^2\dif \xi \Bigr)^\frac{1}{2}
 \Bigl(\int_{\Rset^N} \bigabs{\widehat{g}(\xi)}^2\dif \xi \Bigr)^\frac{1}{2}\\
 &= \tfrac{\alpha}{s} \norm{(-\Delta)^\frac{s}{2} f}_{L^2 (\Rset^N)} \norm{g}_{L^2 (\Rset^N)}\,.
\end{split}
\]
This concludes the proof.
\end{proof}

A variant of Lemma~\ref{lemmaConvolutions} can then be deduced, where the error is estimated in classical \(L^q (\Rset^N)\) and Sobolev norms.

\begin{lemma} \label{convolution0}If \(q > 2\), \(\max \{\frac{2N}{N + 2}, 1\} < r < 2\),
and if
\[
 0< \alpha \le \tfrac{N}{q} - \tfrac{N}{2},
\]
then
\[
 \Bigabs{\int_{\Rset^N} (I_\alpha \ast f) g  - \int_{\Rset^N} fg}\le
 C \alpha \bigl(\norm{f}_{L^q (\Rset^N)} + \norm{\nabla f}_{L^r (\Rset^N)}  \bigr)\norm{g}_{L^2 (\Rset^N)}.
\]
\end{lemma}

\begin{proof}
We shall apply the estimate of Lemma~\ref{lemmaConvolutions}.
We take \(\beta = N (\frac{1}{q} - \frac{1}{2})\), so that, by
 the classical Hardy--Littlewood--Sobolev inequality (see for example \cite{LiebLoss2001}*{theorem 4.3}),
\[
 \norm{I_{\beta} \ast f}_{L^2 (\Rset^N)} \le C \norm{f}_{L^q (\Rset^N)}.
\]
We next take \(s = 1 - N (\frac{1}{r} - \frac{1}{2})\) and we estimate by the Hardy--Littlewood--Sobolev inequality:
\begin{equation*}
\begin{split}
 \norm{(-\Delta)^\frac{s}{2} f}_{L^2}
 &= \norm{(2 \pi |\xi|)^s \Hat{f}(\xi)}_{L^2 (\Rset^N, \dif \xi)}
 = (2 \pi)^s \norm{ \abs{\xi}^{s-1} \, \abs{\xi} \, \Hat{f}(\xi) }_{L^2(\Rset^N, \dif \xi)} \\
  &= \norm{ I_{1 - s} \ast \abs{\nabla f}}_{L^2 (\Rset^N)}
  \le C \norm{\nabla f}_{L^r}.\qedhere
\end{split}
\end{equation*}
\end{proof}

\begin{remark}
The control in terms of $\norm{\nabla f}_{L^r (\Rset^N)}$ might seem unnatural in Lemma~\ref{convolution0}, but it is actually necessary for Lemma~\ref{convolution0} to hold.
Indeed, if we choose a nonzero function $\psi \in C^{\infty}_0(\Rset^N)$ and define $f_n(x)=g_n(x)= e^{2\pi i\, n \eta\cdot x} \, \psi(x)$ for some fixed $\eta \in \R^N \setminus \{0\}$, clearly, $\abs{f_n} = \abs{f_0}$ in \(\Rset^N\) and then the sequence \((f_n)_{n \in \Nset}\) is bounded in \(L^q (\Rset^N)\) for every \(q > 2\). By the translation properties of the Fourier transform, $\widehat{f_n}(\xi)= \widehat{\psi}(\xi- n \eta)$. Now, if \(n^{\alpha_n} \to 0\) as \(n \to \infty\), we have
\begin{multline*}
  \int_{\Rset^N}  \bigl((2 \pi\abs{\xi})^{-\alpha_n} - 1\bigr)  \bigabs{\widehat{f_n}(\xi)}^2 \dif \xi =
  \int_{\Rset^N} \bigl((2\pi \abs{\zeta + n \eta})^{-\alpha_n} - 1\bigr) \bigabs{\widehat{\psi}(\zeta)}^2 \dif \zeta \\
  \to -
  \int_{\Rset^N} \abs{\widehat{\psi}(\zeta)}^2 \dif \zeta<0.
\end{multline*}
Observe that in this case, the sequence $(\abs{\nabla f_n})_{n \in \Nset}$ is not bounded in any $L^r (\Rset^N)$ space.
\end{remark}

\begin{remark}
\label{convolution1}
Given \(u \in H^1 (\Rset^N)\), we set \(f = g = \abs{u}^p\).
By the Sobolev embedding theorem and by the H\"older inequality, we have \(\abs{u}^p \in L^q (\Rset^N)\)
for every \(q >1\) such that \(\frac{1}{q} \ge p (\frac{1}{2} - \frac{1}{N})\) and
\[
\nabla (\abs{u}^p) = p \,\abs{u}^{p-2} u \,\nabla u \ \in L^{r} (\Rset^N), \text{ for each \(r \in [1, \infty)\) such that } \frac{1}{r} \ge \frac{p}{2} - \frac{p - 1}{N},
\]
by H\"older's inequality.
Since \(\frac{1}{p} > 1 - \frac{2}{N}\) we have
\[
  \frac{p}{2} - \frac{p - 1}{N} < \frac{1}{N} + \frac{1}{2} = \frac{N + 2}{2N},
\]
and we are thus in the applicability range of Lemma~\ref{convolution0}, and we have
\[
 \Bigabs{\int_{\Rset^N} \abs{u}^{2 p} - \int_{\Rset^N} \bigl(I_\alpha \ast \abs{u}^p) \abs{u}^p}
 \le C \alpha \Bigl(\int_{\Rset^N} \abs{\nabla u}^2 + \abs{u}^2 \Bigr)^p.
\]
\end{remark}

\subsection{Delocalizing Riesz potentials}
\label{sectionDelocalizingRiesz}

In the r\'egime \(\alpha \to N\), we consider the potential \(\Tilde{I}_\alpha\) defined for
\(x \in \Rset^N \setminus \{0\}\) by
\begin{equation}
\label{eqDefUnnormalizedRieszPotential}
  \Tilde{I}_\alpha (x) = \frac{1}{\abs{x}^{N - \alpha}};
\end{equation}
this potential is related to the Riesz potential as follows
\[
 \Tilde{I}_{\alpha} = \frac{\Gamma(\tfrac{\alpha}{2})\pi^{N/2}2^{\alpha} }{\Gamma(\tfrac{N-\alpha}{2})} I_\alpha,
\]
and, as \(\alpha \to N\),
\[
 \frac{\Gamma(\tfrac{\alpha}{2})\pi^{N/2}2^{\alpha} }{\Gamma(\tfrac{N-\alpha}{2})}
 = \Gamma (\tfrac{N}{2}) \pi^{N/2} 2^{N - 1} (N - \alpha) \bigl(1 + o (1)\bigr).
\]

In the next lemma we give un upper bound for the Riesz potential energy:

\begin{lemma}
\label{lemmaConvEstimate}
Let \(r \in (1, \infty)\).
For every \(\alpha \in (N/r, N)\), if \(f \in L^1 (\Rset^N) \cap L^r (\Rset^N)\) is nonnegative and
\(x \in \Rset^N\),
\[
  (\Tilde{I}_\alpha \ast f) (x) \le \int_{\Rset^N} f +
  C \frac{N - \alpha}{(r \alpha - N)^{1 - 1/r}}
  \Bigl(\int_{\Rset^N} f^r \Bigr)^\frac{1}{r}.
\]
In particular, if the function \(g \in L^1 (\Rset^N)\) is nonnegative, then
\[
  \int_{\Rset^N} (\Tilde{I}_\alpha \ast f) \, g
  \le \biggl(\int_{\Rset^N} f\biggr) \biggl(\int_{\Rset^N} g\biggr)
  + C \frac{N - \alpha}{(r \alpha - N)^{1 - 1/r}}
  \Bigl(\int_{\Rset^N} f^r \Bigr)^\frac{1}{r} \int_{\Rset^N} g.
\]
\end{lemma}
\begin{proof}
Since the function \(f\) is summable and nonnegative, for each \(x \in \Rset^N\),
\[
\begin{split}
 (\Tilde{I}_\alpha \ast f) (x) - \int_{\Rset^N} f
 = \int_{\Rset^N} \biggl(\frac{1}{\abs{x-y}^{N - \alpha}} - 1\biggr) f (y) \dif y
 \le \int_{B_1 (x)} \biggl(\frac{1}{\abs{x-y}^{N - \alpha}} - 1\biggr) f (y) \dif y.
\end{split}
\]
Therefore, by the classical H\"older inequality, we have
\begin{equation}
\label{ineqTIalphaxisyr}
  (\Tilde{I}_\alpha \ast f) (x) - \int_{\Rset^N} f
  \le \Bigl(\int_{B_1 (0)} \Bigl(\frac{1}{\abs{z}^{N - \alpha}} - 1\Bigr)^\frac{r}{r - 1} \Bigr)^{1 - \frac{1}{r}}\Bigl(\int_{\Rset^N} f^r \Bigr)^\frac{1}{r}.
\end{equation}

In order to estimate the first integral, we first perform a radial integration:
\begin{equation}
\label{eqRadialIntegration}
  \int_{B_1 (0)} \Bigl(\frac{1}{\abs{z}^{N - \alpha}} - 1\Bigr)^\frac{r}{r - 1}\dif z
  = \abs{\partial B_1 (0)} \int_0^1 \Bigl(\frac{1 - s^{N - \alpha}}{s^{N - \alpha}}\Bigr)^\frac{r}{r - 1}s^{N - 1} \dif s.
\end{equation}
On the one hand, the latter integral can be bounded by
\begin{equation}
\label{ineqTIalphaxisyrLargealpha}
  \int_0^1 \Bigl(\frac{1 - s^{N - \alpha}}{s^{N - \alpha}}\Bigr)^\frac{r}{r - 1}s^{N - 1} \dif s
  \le \int_0^1 s^{N - (N - \alpha)\frac{r}{r - 1} - 1} \dif s
  = \frac{r - 1}{N - r \alpha}.
\end{equation}
On the other hand, by elementary convexity, we have for each \(s \in (0, 1]\),
\[
1 + (N - \alpha) \ln s
\le \exp \Bigl((N - \alpha) \ln s\Bigr) =
 s^{N - \alpha},
\]
and therefore
\begin{equation}
\label{ineqTIalphaxisyrSmallalpha}
\begin{split}
 \int_0^1 \Bigl(\frac{1 - s^{N - \alpha}}{s^{N - \alpha}}\Bigr)^\frac{r}{r - 1}s^{N - 1} \dif s
 &\le (N - \alpha)^\frac{r}{r - 1} \int_0^1  \Bigl(\ln \frac{1}{s} \Bigr)^\frac{r}{r - 1} s^{N - 1 - \frac{r}{r - 1}(N - \alpha)} \dif s.
\end{split}
\end{equation}
The first inequality follows from \eqref{ineqTIalphaxisyr}, \eqref{eqRadialIntegration}, \eqref{ineqTIalphaxisyrLargealpha} and \eqref{ineqTIalphaxisyrSmallalpha}.

The second inequality is obtained by multiplying the first one
by \(g (x)\) and integrating with respect to \(x \in \Rset^N\).
\end{proof}

Next lemma is concerned with the reversed inequality:

\begin{lemma}
\label{lemmaConv}
Let \(r \in (1, \infty)\), \((\alpha_n)_{n \in \Nset}\) be a sequence in \((N/r, N)\) converging to \(N\),
\((\xi_n)_{n \in \Nset}\) be a sequence in \(\Rset^n\) and let \((f_n)_{n \in \Nset}\) be a bounded sequence of functions in \(L^r (\Rset^N)\). If \((f_n(\cdot - \xi_n))_{n \in \Nset}\) converges strongly to \(f\) in \(L^1 (\Rset^N)\) and
if \((1/(1+|\xi_n|)^{N-\alpha_n})_{n \in \Nset}\) converges to \(\varrho \in [0,1]\), then
\begin{equation} \label{uno}
  \lim_{n \to \infty} (\Tilde{I}_{\alpha_n} \ast f_n) (x) = \varrho \int_{\Rset^N} f \ \mbox{ for any } x \in \Rset^N.
\end{equation}
If moreover \((g_n)_{n \in \Nset}\) is a sequence of functions
that converges to \(g\) strongly in \(L^1 (\Rset^N)\), then
\begin{equation} \label{dos}
\lim_{n \to \infty} \int_{\Rset^N} (\Tilde{I}_{\alpha_n} \ast f_n)\, g_n = \varrho   \Bigl(\int_{\Rset^N} f\Bigr) \Bigl(\int_{\Rset^N} g \Bigr).
\end{equation}
\end{lemma}

Lemma~\ref{lemmaConv} gives a good idea of the validity of the reversed bound of Lemma~\ref{lemmaConvEstimate}. Indeed, $\Tilde{I}_{\alpha_n} \ast f_n(x) \to \int_{\R^N} f$ for \emph{converging sequences} $f_n$, but it can fail for sequences of functions given by translations in the $x$-variable. We shall prove later that the least energy nodal solutions behave as two signed bumps whose distance diverges. As we shall see, this makes our proofs more involved in the case $\alpha$ close to $N$.

\begin{proof}[Proof of Lemma~\ref{lemmaConv}]
We can assume that $x=0$, by making a suitable translation.
We rewrite for each \(n \in \Nset\) the quantities appearing in \eqref{uno} in integral form
\[
(\Tilde{I}_{\alpha_n} \ast f_n) (0) - \varrho \int_{\Rset^N} f_n
= \int_{\Rset^N}  f_n (y) \biggl(\frac{1}{\abs{y}^{N - \alpha_n}} - \rho\biggr)\dif y.
\]
Given \(\delta \in (0, 1)\), by the H\"older inequality  and by Lemma~\ref{lemmaConvEstimate}, we first have
\[
\begin{split}
  \Bigabs{\int_{B_\delta}f_n ( y) \biggl(\frac{1}{\abs{y}^{N - \alpha_n}} - \rho\biggr)\dif y}
  &\le \int_{B_\delta} \abs{f_n ( y)}\, \biggl(\frac{1}{\abs{y}^{N - \alpha_n}} + \rho\biggr)\dif y\\
  &\le (1+ \varrho) \abs{B_\delta}^{1 - \frac{1}{p}} \| f_n\|_{L^p(\Rset^N)} + C(N - \alpha_n) \| f_n\|_{L^r(\RN)},
\end{split}
\]
Next, we write
\[
 \int_{\Rset^N \setminus B_\delta}f_n (y) \biggl(\frac{1}{\abs{y}^{N - \alpha_n}} - \rho\biggr)\dif y
 = \int_{\Rset^N \setminus B_\delta (\xi_n) }f_n (z - \xi_n) \biggl(\frac{1}{\abs{z  - \xi_n}^{N - \alpha_n}} - \rho\biggr) \dif y.
\]
We observe that for every \(z \in \Rset^N \),
\begin{equation}
\label{eqLimitRieszTranslated}
  \lim_{n \to \infty} \Bigl(\frac{1}{\abs{z - \xi_n}^{N - \alpha_n}} - \rho\Bigr) \chi_{B_\delta (\xi_n)} = 0.
\end{equation}
Indeed, we have by the triangle inequality, on the one hand,
\[
  \abs{z - \xi_n} \le \abs{z} + \abs{\xi_n} \le (1 + \abs{z})(1 + \abs{\xi_n})
\]
and on the other hand, if \(z \in \Rset^N \setminus B_\delta (\xi_n)\), we have
\[
 \abs{z - \xi_n}
 \ge (1 - \lambda) \delta + \lambda (\abs{\xi_n} - \abs{z})
 = \lambda (\abs{\xi_n} + 1).
\]
with \(\lambda = \delta/(1 + \abs{z} + \delta)\).
Therefore it follows that, if \(z \not \in B_{\delta} (\xi_n)\)
\[
 \Bigl(\frac{1}{\abs{z} + 1}\Bigr)^{N - \alpha_n} \frac{1}{(1 + \abs{\xi_n})^{N - \alpha_n}} \le \frac{1}{\abs{z - \xi_n}^{N - \alpha_n}} \le \Bigl(\frac{1 + \abs{z} + \delta}{\delta}\Bigr)^{N - \alpha_n} \frac{1}{(1 + \abs{\xi_n})^{N - \alpha_n}},
\]
and the limit \eqref{eqLimitRieszTranslated} follows.
Therefore, by Lebesgue's dominated convergence theorem,
\[
 \lim_{n \to \infty} \int_{\Rset^N \setminus B_\delta (\xi_n) }f_n (z - \xi_n) \biggl(\frac{1}{\abs{z - \xi_n}^{N - \alpha_n}} - \rho\biggr) \dif y = 0.
\]
Since $\delta>0$ is arbitrary, the first conclusion \eqref{uno} follows.

\medskip In order to prove \eqref{dos}, we pass to a subsequence so that $g_n \to g$ almost everywhere in \(\Rset^N\) as \(n \to \infty\) and for each \(n \in \Nset\), $|g_n| \leq h$ in \(\Rset^N\) for some $h \in L^1(\Rset^N)$.
By \eqref{uno}, the sequence $(\Tilde{I}_{\alpha_n} \ast f_n)_{n \in \Nset}$ converges to the constant $\varrho \int_{\Rset^N} f$ everywhere in \(\Rset^N\).
By Lemma~\ref{lemmaConvEstimate}, the sequence $(\Tilde{I}_{\alpha_n} \ast f_n)_{n \in \Nset}$ is uniformly bounded over \(\Rset^N\), so that Lebesgue's dominated convergence theorem applies and brings the conclusion.
\end{proof}

\section{Proof of Theorem~\ref{theoremMain0}}
\label{sectionProofTheoremMain0}

This section is devoted to the proof of Theorem~\ref{theoremMain0}. As a first step, in the next proposition we show that least energy nodal solutions are asymptotically odd with respect to a hyperplane.

\begin{proposition} \label{puf}
Let \(u_\alpha\) be a family of solutions to \eqref{eqChoquard} that changes sign and satisfying \(I_\alpha (u_\alpha) \le 2 c_\alpha \), then
\[
 \lim_{\alpha \to 0}   \ \ \inf_{\xi^+, \xi^- \in \Rset^N} \norm{u_\alpha - (U (\cdot - \xi^+) - U (\cdot - \xi^-))}_{H^1 (\Rset^N)} = 0.
\]
Moreover, for \(\alpha \in (0, N)\) small enough there exist \(\xi_\alpha^+, \xi_\alpha^-\)
such that
\[
 \norm{u_\alpha - (U (\cdot - \xi_\alpha^+) - U (\cdot - \xi_\alpha^-))}_{H^1 (\Rset^N)}
 =\inf_{\xi^+, \xi^- \in \Rset^N} \norm{u_\alpha - (U (\cdot - \xi^+) - U (\cdot - \xi^-))}_{H^1 (\Rset^N)}
\]
and
\begin{equation}%
\label{distance}
 \lim_{\alpha \to 0} \, \abs{\xi_\alpha^+ - \xi_\alpha^-} = + \infty.
\end{equation}

\end{proposition}

By Proposition~\ref{propositionControlGroundNodal}, if \(u_\alpha\) is a least energy nodal solution we have \(I_{\alpha}(u_\alpha) = c_\alpha^\nod < 2 c_{\alpha}\) and
thus \(u_\alpha\) satisfies the assumption of Proposition~\ref{puf}.

\begin{proof}[Proof of Proposition~\ref{puf}]
The proof relies on some preliminary results which are stated in form of claims.
\begin{claim}
\label{claim0Gst}
One has
\[
 \lim_{\alpha \to 0} c_\alpha^\gst = \gamma_{2 p}.
\]
\end{claim}

\begin{proofclaim}

Given $u \in H^1(\Rset^N) \setminus \{0\}$, by Lemma~\ref{convolution0} we have, as \(\alpha \to 0\),
\begin{align*}
  c_{\alpha}^{\gst} \leq \max_{t>0} J_{\alpha}( t u) = \frac{t^2}{2} \Big (\int_{\Rset^N} |\nabla u|^2 + \abs{u}^2 \Big) - \frac{t^{2p}}{2p} \int_{\Rset^N} \bigl(I_\alpha \ast \abs{u}^p\bigr)\abs{u}^p  \\ \to \max_{t>0} \frac{t^2}{2} \Big (\int_{\Rset^N} |\nabla u|^2 + \abs{u}^2 \Big )- \frac{t^{2p}}{2p} \int_{\Rset^N} \abs{u}^{2p} = \max_{t>0} \Phi_{2p}(tu).
\end{align*}
Taking  the infimum with respect to \(u \in H^1(\Rset^N) \setminus \{0\}\), we deduce that
\begin{equation}
\label{ineqClaim1Limsup}
\limsup_{\alpha \to 0}c_{\alpha}^{\gst} \leq \gamma_{2p}.
\end{equation}
Since
\[
  c_\alpha^{\gst} = \Big (\frac 1 2 - \frac{1}{2p} \Big) \Big (\int_{\Rset^N} |\nabla U_\alpha|^2 + U_\alpha^2 \Big),
\]
the groundstate solution $U_\alpha$ remains bounded in $H^1(\Rset^N)$ as $\alpha \to 0$. By Lemma~\ref{convolution0} again, we also have
\[
c_{\alpha}^{\gst} =J_{\alpha}(U_\alpha) = \max_{t>0} J_{\alpha}(t U_\alpha) = \max_{t>0} \Phi_{2p}(tU_\alpha) + o (\alpha),
\]
as \(\alpha \to 0\).
This implies
\begin{equation}
\label{ineqClaim1Liminf}
  \liminf_{\alpha \to 0}
    c_{\alpha}^{\gst}
  \ge \gamma_{2 p}.
\end{equation}
The claims follows from the combination of the inequalities \eqref{ineqClaim1Limsup} and \eqref{ineqClaim1Liminf}.
\end{proofclaim}

\begin{claim}
\label{claim0Bounded}
The family \((u_\alpha)_{\alpha > 0}\) is bounded in \(H^1 (\Rset^N)\) as \(\alpha \to 0\).

%\[
%  \limsup_{\alpha \to 0} \int_{\Rset^N}\abs{\nabla u_\alpha}^2 + \abs{u_\alpha}^2
% <\infty.
%\]

\end{claim}

\begin{proofclaim}
By assumption, we have for each \(\alpha \in (0, N)\) \(I_{\alpha}(u_\alpha) \le 2 c_{\alpha} \).
Since for each \(\alpha \in (0, N)\), we also have \(\dualprod{I_\alpha'(u_\alpha)}{u_\alpha}=0\), we deduce
\[
 \int_{\Rset^N} \abs{\nabla u_\alpha}^2 + \abs{u_\alpha}^2
 = \frac{2 p}{p - 1} I_\alpha (u_\alpha) \le \frac{2 p}{p - 1} 2 c_\alpha.
\]
The claim follows then from Claim~\ref{claim0Gst}.
\end{proofclaim}

\begin{claim}%
\label{claimTheorem0SignChanges}%
One has
\begin{equation*}
 \liminf_{\alpha \to 0} \int_{\Rset^N} \abs{\nabla u_\alpha^\pm}^2 + \abs{u_\alpha^\pm}^2 =
 \liminf_{\alpha \to 0} \int_{\Rset^N} \bigl(I_\alpha \ast \abs{u_\alpha}^p\bigr) \abs{u_\alpha^\pm}^p > 0.
\end{equation*}
\end{claim}

\begin{proofclaim}
We recall that the optimal Hardy--Littlewood--Sobolev inequality
\citelist{%
  \cite{LiebLoss2001}*{Theorem 4.3}%
  \cite{Lieb1983}*{Theorem 3.1}%
} %
states that if \(f, g \in L^\frac{2N}{N + \alpha} (\Rset^N)\), then
\begin{equation}
\label{ineqOptimalHLSRiesz}
  \int_{\Rset^N} (I_\alpha \ast f) g
  \le C_{N, \alpha} \Bigl(\int_{\Rset^N} \abs{f}^\frac{2N}{N + \alpha} \Bigr)^\frac{N + \alpha}{2 N}
   \Bigl(\int_{\Rset^N} \abs{g}^\frac{2N}{N + \alpha} \Bigr)^\frac{N + \alpha}{2 N},
\end{equation}
and the optimal constant \(C_{N, \alpha}\) is given in terms of the gamma function \(\Gamma\) by
\begin{equation}
\label{ineqOptimalHLSRieszConstant}
  C_{N, \alpha} = \frac{\Gamma (\frac{N - \alpha}{2})}{2^\alpha \pi^{\alpha/2} \Gamma (\frac{N + \alpha}{2}) }
  \biggl(\frac{\Gamma(\frac{N}{2})}{\Gamma (N)}\biggr)^\frac{\alpha}{N}.
\end{equation}
If \(p > 2\), we observe that, by the Hardy--Littlewood--Sobolev inequality \eqref{ineqOptimalHLSRiesz} and by the classical Sobolev inequality
\[
\begin{split}
 \int_{\Rset^N} \abs{\nabla u_\alpha^\pm}^2 + \abs{u_\alpha^\pm}^2
 &= \int_{\Rset^N} \bigl(I_\alpha \ast \abs{u_\alpha}^p\bigr) \abs{u_\alpha^\pm}^p\\
 &\le C C_{N, \alpha} \Bigl(\int_{\Rset^N} \abs{\nabla u_\alpha}^2 + \abs{u_\alpha}^2 \Bigr)^\frac{p}{2}
 \Bigl(\int_{\Rset^N} \abs{\nabla u_\alpha^\pm}^2 + \abs{u_\alpha^\pm}^2 \Bigr)^\frac{p}{2}.
\end{split}
\]
We deduce therefrom that
\[
 1 \le C C_{N, \alpha} \Bigl(\int_{\Rset^N} \abs{\nabla u_\alpha}^2 + \abs{u_\alpha}^2 \Bigr)^\frac{p}{2}
 \Bigl(\int_{\Rset^N} \abs{\nabla u_\alpha^\pm}^2 + \abs{u_\alpha^\pm}^2 \Bigr)^\frac{p - 2}{2}.
\]
In view of \eqref{ineqOptimalHLSRieszConstant},
we have
\[
 \lim_{\alpha \to 0} C_{N, \alpha} = 1,
\]
and the constant \(C_{N, \alpha}\) remains thus bounded as \(\alpha \to 0\).
Since the family \((u_\alpha)_{\alpha \in (0, N)}\) is bounded in \(H^1 (\Rset^N)\) in view of Claim~\ref{claim0Bounded},
we have, if \(p > 2\),
\[
 \liminf_{\alpha \to 0} \int_{\Rset^N}\abs{\nabla u_\alpha^\pm}^2 + \abs{u_\alpha^\pm}^2
 > 0.
\]
The claim is thus proved in the case \(p > 2\).

If \(p = 2\) we adapt the strategy of \cite{GhimentiMorozVanSchaftingen}. Since for each \(\alpha \in (0, N)\)
\[
 \int_{\Rset^N} \abs{\nabla u_\alpha}^2 + \abs{u_\alpha}^2
 = \int_{\Rset^N} \bigl(I_\alpha \ast \abs{u_\alpha}^2\bigr) \abs{u_\alpha^\pm}^2
 \le C C_{N, \alpha} \Bigl(\int_{\Rset^N} \abs{\nabla u_\alpha}^2 + \abs{u_\alpha}^2 \Bigr)^2,
\]
the functions \(u_\alpha\) stay away from \(0\) as \(\alpha \to 0\):
\[
  \liminf_{\alpha \to 0} \int_{\Rset^N}\abs{\nabla u_\alpha}^2 + \abs{u_\alpha}^2
 > 0.
\]
Without loss of generality, we can assume that
\[
 \liminf_{\alpha \to 0} \int_{\Rset^N} \abs{\nabla u_\alpha^+}^2 + \abs{u_\alpha^+}^2 > 0.
\]
We are going to prove that
\[
 \liminf_{\alpha \to 0} \int_{\Rset^N} \abs{\nabla u_\alpha^-}^2 + \abs{u_\alpha^-}^2 > 0.
\]
Otherwise, there would exist a sequence \((\alpha_n)_{n \in \Nset}\) in \((0, N)\) converging to \(0\)
such that
\[
 \lim_{n \to \infty} \int_{\Rset^N}\abs{\nabla u_{\alpha_n}^-}^2 + \abs{u_{\alpha_n}^-}^2 = 0.
\]
We could then define for each \(\alpha \in (0, N)\) the normalized negative part of \(u_\alpha\) by
\[
 v_\alpha = \frac{u_\alpha^-}{\norm{u_\alpha^-}_{H^1 (\Rset^N)}}.
\]
By the Hardy--Littlewood--Sobolev and by the Sobolev inequalities, for every \(\alpha \in (0, N)\),
\[
  \int_{\Rset^N} \bigl(I_{\alpha} \ast \abs{u_{\alpha}^-}^2\bigr)\abs{u_{\alpha}^-}^2
  \le C \Bigl(\int_{\Rset^N} \abs{\nabla u_{\alpha}^-}^2 + \abs{u_{\alpha}^-}^2 \Bigr)^2,
\]
and therefore, we would write
\[
\begin{split}
 1 &= \frac{\displaystyle \int_{\Rset^N} \bigl(I_{\alpha_n} \ast \abs{u_{\alpha_n}}^2\bigr)\abs{u_{\alpha_n}^-}^2}
 {\displaystyle \int_{\Rset^N} \abs{\nabla u_{\alpha_n}^-}^2 + \abs{u_{\alpha_n}^-}^2}
 =\lim_{n \to \infty} \frac{\displaystyle \int_{\Rset^N} \bigl(I_{\alpha_n} \ast \abs{u_{\alpha_n}^+}^2\bigr)\abs{u_{\alpha_n}^-}^2}
 {\displaystyle \int_{\Rset^N} \abs{\nabla u_{\alpha_n}^-}^2 + \abs{u_{\alpha_n}^-}^2}\\
 &=\lim_{n \to \infty} \int_{\Rset^N} \bigl(I_{\alpha_n} \ast \abs{u_{\alpha_n}^+}^2\bigr)\abs{v_{\alpha_n}^-}.
\end{split}
\]
Now, by taking into account Remark~\ref{convolution1}, we would apply Lemma~\ref{convolution0} with $g=\abs{v_{\alpha_n}}^2$ to obtain
\[
 \lim_{n \to \infty} \Bigl(\int_{\Rset^N} \bigl(I_{\alpha_n} \ast \abs{v_{\alpha_n}}^2 \bigr)\abs{u_{\alpha_n}^+}^2
 -\int_{\Rset^N} \abs{u_{\alpha_n}^+}^2 \abs{v_{\alpha_n}}^2 \Bigr)= 0.
\]
But, by construction, \(\abs{u_\alpha^+}^2 \abs{v_\alpha}^2 = 0\) almost everywhere
in \(\Rset^N\), and we would thus reach a contradiction.
\end{proofclaim}

With those results in hand, we are now in condition to prove Proposition~\ref{puf}. By Lemma~\ref{convolution0} and Remark~\ref{convolution1}, if \(\alpha \to 0\), then the family \((u_{\alpha})_{\alpha \in (0, N)}\) forms a Palais--Smale sequence for the limit problem \eqref{limit0}. Moreover, by Claim~\ref{claim0Gst} and by our assumption
\[
  \limsup_{\alpha \to 0} I_{\alpha}(u_{\alpha})
  \le \limsup_{\alpha \to 0} 2 c^{gst}_\alpha
  \le 2 \gamma_{2p}.
\]
In view of  \eqref{strictOtherSolutions0}, Lemma~\ref{PS} implies that $m \leq 2$ and $u_i = \pm U$.
By Claim~\ref{claimTheorem0SignChanges}, we can assume without loss of generality that
\begin{equation} \label{justone} u_\alpha = U(\cdot-\tilde{\xi}^+_\alpha) - U(\cdot-\tilde{\xi}^-_\alpha) + o(1), \ \ |\tilde{\xi}^+_\alpha - \tilde{\xi}^-_\alpha| \to +\infty. \end{equation}
We observe that, by Fatou's lemma, for each \(\alpha \in (0, N)\),
\[
  \liminf_{\abs{\xi^+} + \abs{\xi^-} \to +\infty}
  \norm{u_\alpha - (U (\cdot - \xi^+) - U (\cdot - \xi^-))}_{H^1 (\Rset^N)}
  \ge \min \{\norm{u_\alpha}_{H^1 (\Rset^N)}, \norm{U}_{H^1 (\Rset^N)} \},
\]
By Claim~\ref{claimTheorem0SignChanges}, the right-hand side stays away from \(0\) as \(\alpha \to 0\).
When \(\alpha \in (0, N)\) is small enough, by the first part of the claim, the function
\[
  (\xi^+, \xi^-) \in \Rset^N \times \Rset^N
  \longmapsto \norm{u_\alpha - (U (\cdot - \xi^+) - U (\cdot - \xi^-))}_{H^1 (\Rset^N)}
\]
achieves thus its minimum at some pair of vectors \((\xi_{\alpha}^+, \xi_{\alpha}^-) \in \Rset^N \times \Rset^N\). By \eqref{justone}, that minimum goes to $0$, that is,
$$u_\alpha = U(\cdot-{\xi}^+_\alpha) - U(\cdot-{\xi}^-_\alpha) + o(1).$$
We note that, again by \eqref{justone}, $|\tilde{\xi}^\pm_\alpha - {\xi}^\pm_\alpha| \to 0$. Finally,  Lemma \ref{PS} implies that \(\abs{\xi_\alpha^+ - \xi_\alpha^-} \to +\infty\) as \(\alpha \to 0\).
\end{proof}

\medbreak

\begin{proof}%
[Proof of Theorem~\ref{theoremMain0}]
Since the Choquard equation \eqref{eqChoquard} is invariant under translations and rotations, we can assume that for each \(\alpha \in (0, N)\) sufficiently close to \(0\), \(\xi_{\alpha, +} = 0\)
and $\xi_{\alpha,-} = \xi_{\alpha} = (m_\alpha, 0, \dotsc, 0)$, for some $m_\alpha \to
+ \infty$.
We define then \(R_\alpha\) to be the orthogonal reflection that sends \(0\) to \(\xi_\alpha\),
that is, for each \( x = (x_1, \dotsc, x_N) \in \Rset^N\),
\[
 R_{\alpha} (x) = (m_\alpha -x_1, x_2, \dotsc, x_N).
\]
We set \(v_\alpha = u_\alpha + \Breve{u}_\alpha\), where \(\Breve{u}_\alpha = u_\alpha \circ R_\alpha\).
We define also for every such \(\alpha \in (0, N)\) the half-space
\[
 \Omega_\alpha = \Bigl\{ x=(x_1, \dotsc, x_N)  \in \Rset^N \st x_1 < \frac{m_\alpha}{2}\Bigr\}.
\]
By construction, \(v_\alpha \circ R_\alpha = v_\alpha\), and
thus the function \(v_\alpha\) is even with respect to \(\partial \Omega_\alpha\).

Our purpose is to show that for $\alpha>0$ small enough, $v_\alpha=0$, from which  Theorem~\ref{theoremMain0} will follow immediately. The proof will rely on some preliminary results:

\begin{step}
\label{claimTheorem0Orthogonality}
For any direction \(\zeta \in \Rset^N\),
\begin{align*}
  \scalprod{v_\alpha}{\partial_\zeta U}_{H^1(\Rset^N)} &=0,&
  &\scalprod{v_\alpha}{\partial_\zeta U}_{H^1 (\Omega_\alpha)}
   = o \bigl(\norm{v_\alpha}_{H^1 (\Omega_\alpha)}\big),
\end{align*}
as \(\alpha \to 0\).
\end{step}

The function
$$ (\xi^+, \xi^-) \in \Rset^N \times \Rset^N \mapsto \norm{u_\alpha - (U (\cdot - \xi^+) - U (\cdot - \xi^-))}_{H^1 (\Rset^N)}^2$$
attains a minimum at $(0, \xi_{\alpha})\in \Rset^N \times \Rset^N$. Differentiating with respect to the variable $\xi^+$ in the direction $\zeta \in \Rset^N$, we obtain
\begin{equation}%
\label{eqTheorem0torecall}
 \scalprod{u_\alpha}{\partial_\zeta U}_{H^1 (\Rset^N)} = 0.
\end{equation}
Reasoning in an analogous way on the variable $\xi^-$, we get
\[
 \scalprod{u_\alpha}{\partial_{\zeta} U (\cdot - \xi_\alpha)}_{H^1 (\Rset^N)} = 0.
\]
We now observe that
\begin{enumerate}
\item if $\zeta=\xi_{\alpha}$, then $\partial_\zeta U \circ R_{\alpha}(x)= -\partial_{\zeta} U (\cdot - \xi_\alpha)$,
\item if $\zeta \cdot \xi_{\alpha} = 0$, then $\partial_\zeta U \circ R_{\alpha}(x)= \partial_{\zeta} U (\cdot - \xi_\alpha)$,
\end{enumerate}
so that, in any case,
\[
 \scalprod{u_\alpha}{\partial_{\zeta} U (\cdot - \xi_\alpha)}_{H^1 (\Rset^N)} = \pm \scalprod{\Breve{u}_\alpha}{\partial_{\zeta} U}_{H^1 (\Rset^N)}=0.
\]
This, together with \eqref{eqTheorem0torecall}, concludes the proof of the first assertion. The second follows since $\norm{\partial_\zeta U}_{H^1(\Rset^N \setminus \Omega_\alpha)}=o(1)$.
\medbreak

\begin{step}
The function \(v_\alpha\) satisfies the linear equation
\begin{equation} \label{extra0}
 \mathcal{L}_\alpha v_\alpha = 0,
\end{equation}
where the linear operator \(\mathcal{L}_\alpha\) is defined by
\[
 \mathcal{L}_\alpha v = -\Delta v + v - (I_\alpha \ast G_\alpha v) H_\alpha
 - (I_\alpha \ast K_\alpha) L_\alpha v,
\]
with the functions \(G_\alpha\), \(H_\alpha\), \(K_\alpha\) and \(L_\alpha\) being given by
\begin{align*}
 G_\alpha
 &=  \begin{cases}
   \frac{\abs{u_\alpha}^p - \abs{\Breve{u}_\alpha}^p}{u_\alpha + \Breve{u}_\alpha}& \text{where \(u_\alpha \ne - \Breve{u}_\alpha\)},\\
   p \,\abs{u_\alpha}^{p - 2} u_\alpha & \text{elsewhere},
  \end{cases}\\
 H_\alpha
 &=\frac{1}{2} \bigl(\abs{u_\alpha}^{p - 2} u_\alpha
 - \abs{\Breve{u}_\alpha}^{p - 2} \Breve{u}_\alpha\bigr),\\
  K_\alpha &=\frac{1}{2} \bigl(\abs{u_\alpha}^p + \abs{\Breve{u}_\alpha}^p\bigr),\\
  L_\alpha &= \begin{cases}
   \frac{\abs{u_\alpha}^{p - 2} u_\alpha
 + \abs{\Breve{u}_\alpha}^{p - 2} \Breve{u}_\alpha}{u_\alpha + \Breve{u}_\alpha} & \text{where \(u_\alpha \ne -\Breve{u}_\alpha\)},\\
   (p - 1) \abs{u_\alpha}^{p - 2} & \text{elsewhere}.
  \end{cases}
\end{align*}
\end{step}

By definition of \(v_\alpha\) in terms of \(u_\alpha\) and by the equation \eqref{eqChoquard} satisfied by \(u_\alpha\), the function \(v_\alpha\) obeys the equation
\begin{equation} \label{eqTheorem0v}
 -\Delta v_\alpha + v_\alpha = (I_\alpha \ast \abs{u_\alpha}^p) \abs{u_\alpha}^{p - 2} u_\alpha \
+ (I_\alpha \ast \abs{\Breve{u}_\alpha}^p) \abs{\Breve{u}_\alpha}^{p - 2} \Breve{u}_\alpha .
\end{equation}
We observe that
\begin{multline*}
 (I_\alpha \ast \abs{u_\alpha}^p) \abs{u_\alpha}^{p - 2} u_\alpha
+  (I_\alpha \ast \abs{\Breve{u}_\alpha}^p) \abs{\Breve{u}_\alpha}^{p - 2} \Breve{u}_\alpha\\
 =\frac{1}{2} \bigl(I_\alpha \ast (\abs{u_\alpha}^p + \abs{\Breve{u}_\alpha}^p)\bigr) (\abs{u_\alpha}^{p - 2} u_\alpha
+  \abs{\Breve{u}_\alpha}^{p - 2} \Breve{u}_\alpha)\\
 + \frac{1}{2}\bigl(I_\alpha \ast (\abs{u_\alpha}^p - \abs{\Breve{u}_\alpha}^p) \bigr) (\abs{u_\alpha}^{p - 2} u_\alpha
-  \abs{\Breve{u}_\alpha}^{p - 2} \Breve{u}_\alpha).
\end{multline*}

\medskip

\begin{step}
Conclusion of the proof of Theorem~\ref{theoremMain0}.
\end{step}

As commented above, Theorem~\ref{theoremMain0} follows if we show that $v_{\alpha}=0$ for $\alpha > 0$ small enough. We assume by contradiction that there is a sequence \((\alpha_n)_{n \in \Nset}\) in \((0, N)\) that converges to \(0\) such that for every \(n \in \Nset\), \(v_{\alpha_n} \ne 0\).
For each \(n \in \Nset\), we define the normalized sequence
\[
  w_n = \frac{v_{\alpha_n}}{\norm{v_{\alpha_n}}_{H^1 (\Rset^N)}}.
\]
Without loss of generality, the sequence \((w_n)_{n \in \Nset}\) converges weakly in \(H^1(\Rset^N)\) to some function \(w \in H^1 (\Rset^N)\).

By Proposition~\ref{puf}, we have the following convergences:
\begin{equation}
\label{convNonlinPotentials}
\begin{aligned}
 G_\alpha- p\, (U^{p-1} - U^{p-1}(\cdot - \xi_\alpha)) & \to 0
    & &\text{in } L^q(\Rset^N), &&\text{if } \tfrac{1}{2} - \tfrac{1}{N} \le \tfrac{p - 1}{q} \le \tfrac{1}{2},\\
 H_\alpha - (U^{p-1} - U^{p-1}(\cdot - \xi_\alpha))  & \to 0
    & &\text{in } L^q(\Rset^N), &&\text{if } \tfrac{1}{2} - \tfrac{1}{N} \le \tfrac{p - 1}{q} \le \tfrac{1}{2},\\
 K_\alpha -( U^p + U^p(\cdot- \xi_\alpha)) & \to 0
     & &\text{in } L^q(\Rset^N), &&\text{if } \tfrac{1}{2} - \tfrac{1}{N} \le \tfrac{p}{q} \le \tfrac{1}{2},\\
 L_\alpha - (p-1) ( U^{p-2} + U^{p-2}(\cdot- \xi_\alpha))& \to 0
    & & \text{in } L^q(\Rset^N),  &&\text{if } \tfrac{1}{2} - \tfrac{1}{N} \le \tfrac{p - 2}{q} \le \tfrac{1}{2}.
\end{aligned}
\end{equation}

If we test the equation \(\mathcal{L}_{\alpha} w_\alpha\), against the function \(\varphi \in C_0^{\infty}(\Rset^N)\), we have
\begin{equation*}
\int_{\Rset^N} (\nabla w_\alpha \cdot \nabla \varphi + w_\alpha \varphi)
 = \int_{\Rset^N} (I_\alpha \ast G_\alpha w_\alpha) H_\alpha \varphi
  + (I_\alpha \ast K_\alpha ) L_\alpha w_\alpha\varphi
\end{equation*}
We now apply Lemma~\ref{convolution0} first to \(f= G_\alpha w_\alpha\) and \(g = H_\alpha \varphi\)
and next to $f= K_\alpha$ and \(g = L_\alpha w_\alpha \varphi\).
As in Remark~\ref{convolution1}, the boundedness of $\nabla f$ in $L^r(\R^N)$ for some $r>1$ follows from the Sobolev and H\"older inequalities. In combination with the asymptotic behavior of $G_\alpha$, $H_\alpha$, $K_\alpha$, and $L_\alpha$ in \eqref{convNonlinPotentials}, we deduce that $w$ is a weak solution of
the equation
\[
 -\Delta w + w = (2 p - 1) U^{2 p - 2} w.
\]
By Step~\ref{claimTheorem0Orthogonality} and the nondegeneracy of the limiting problem \eqref{limit0} we have \(w = 0\).

For each \(n \in \Nset\), we now test the equation \(\mathcal{L}_{\alpha_n} v_{\alpha_n} = 0\) against $v_{{\alpha_n}}$ and divide by $\norm{v_{{\alpha_n}}}^2_{H^1 (\Rset^N)}$, to obtain
\begin{equation}
\label{limit1a}
\begin{split}
 1 &=  \int_{\Rset^N} \abs{\nabla w_n}^2 + \abs{w_n}^2\\
 &=  \int_{\Rset^N} \bigl(I_{\alpha_n} \ast (G_{\alpha_n} w_n)\bigr) H_{\alpha_n} w_n
 + \bigl(I_{\alpha_n} \ast K_{\alpha_n}\bigr) L_{\alpha_n} \abs{w_n}^2.
\end{split}
\end{equation}
By Lemma~\ref{convolution0}, on the other hand, we have as \(n \to \infty\)
\begin{equation}
\label{limit10}
\begin{split}
 \int_{\Rset^N} \bigl(I_{\alpha_n} \ast (G_{\alpha_n} w_n)\bigr) &H_{\alpha_n} w_n
 + \bigl(I_{\alpha_n} \ast K_{\alpha_n}\bigr) L_{\alpha_n} \abs{w_n}^2\\
 &= \int_{\Rset^N} \bigl(G_{\alpha_n} H_{\alpha_n}
 + K_{\alpha_n} L_{\alpha_n}\bigr) \abs{w_n}^2 + o(1)\\
 & =2 \int_{\Omega_{\alpha_n}} \bigl(G_{\alpha_n} H_{\alpha_n}
 + K_{\alpha_n} L_{\alpha_n}\bigr) \abs{w_n}^2 + o(1) \\
 &= 2 (2 p - 1) \int_{\Rset^N} \abs{U}^{2p - 2} \abs{w_n}^2 + o(1),
\end{split}
\end{equation}
in view of \eqref{convNonlinPotentials}.
Since the sequence \((w_n)_{n \in \Nset}\) converges weakly to $0$ in \(H^1 (\Rset^N)\), we have in view of Rellich's compactness theorem and the decay of \(U\) at infinity,
\begin{equation}
\label{limit1b}
 \lim_{n \to \infty} \int_{\Rset^N} \abs{U}^{2p - 2} \abs{w_n}^2 = 0,
\end{equation}
 By taking into account \eqref{limit1a}, \eqref{limit10} and \eqref{limit1b} are in contradiction. Hence \(v_\alpha = 0\) for \(\alpha\) close enough to \(0\). The proof of Theorem~\ref{theoremMain0} is thus complete.
\end{proof}

\section{Proof of Theorem~\ref{theoremMainN}}
\label{sectionProofTheoremMainN}

We now turn our attention to the proof of Theorem~\ref{theoremMainN}.
The main difficulty with respect to Theorem \ref{theoremMain0} comes from the fact that the asymptotics of Riesz potential energy
are not as accurate when \(\alpha \to N\) as in the case \(\alpha \to 0\).
This requires additional steps in the proof.

To alleviate the notations, we define \(\underline{\alpha} =  \max \{ 0, (N - 2) p - N \}\).
For \(\alpha \in (\underline{\alpha}, N)\), we first set
\[
 \Tilde{u}_\alpha = (A_\alpha)^\frac{1}{2p - 2} u_\alpha,
\]
where \(A_\alpha\) is the normalizing constant in the Riesz potential \(I_\alpha\)
coming from \eqref{eqRiesz}.
The function \(\Tilde{u}_{\alpha}\) satisfies then the equation
\begin{equation} \label{cov}
 -\Delta \Tilde{u}_\alpha + \Tilde{u}_\alpha
 = \bigl(\Tilde{I}_\alpha \ast \abs{\Tilde{u}_\alpha}^p\bigr) \abs{\Tilde{u}_\alpha}^{p - 2} \Tilde{u}_\alpha,
\end{equation}
with the unnormalized Riesz potential \(\Tilde{I}_\alpha\) that was defined in \eqref{eqDefUnnormalizedRieszPotential}.
We let \(\Tilde{J}_\alpha\), \(\Tilde{c}_\alpha^\gst\) and \(\Tilde{c}_\alpha^\nod\)
denote the corresponding functional, groundstate and least energy  nodal solution levels.

In the next proposition we prove an analogue to Proposition~\ref{puf} of the previous section.

\begin{proposition}
\label{pufii}
If \(\Tilde{u}_\alpha\) are least energy nodal solutions of \eqref{cov}, then
\[
 \lim_{\alpha \to N} \ \ \inf_{\xi^+, \xi^- \in \Rset^N} \norm{\Tilde{u}_\alpha - (V (\cdot - \xi^+) - V (\cdot - \xi^-))}_{H^1 (\Rset^N)} = 0,
\]
where \(V=V_2\) denotes a groundstate of \eqref{limit1} for $\mu=2$.
Moreover, for \(\alpha \in (\underline{\alpha}, N)\) close enough to \(N\) there exists vectors \(\xi_\alpha^+, \xi_\alpha^-\) such that
\begin{equation*}
 \norm{\Tilde{u}_\alpha - (V (\cdot - \xi_\alpha^+) - V(\cdot - \xi_\alpha^-))}_{H^1 (\Rset^N)}
 =\inf_{\xi^+, \xi^- \in \Rset^N} \norm{\Tilde{u}_\alpha - (V (\cdot - \xi^+) - V (\cdot - \xi^-))}_{H^1 (\Rset^N)};
\end{equation*}
moreover they satisfy the asymptotics
\begin{equation*}
 \lim_{\alpha \to N} \abs{\xi_\alpha^+ - \xi_\alpha^-} = \infty,
 \qquad \text{and}\qquad
 \lim_{\alpha \to N} \abs{\xi_\alpha^+ - \xi_\alpha^-}^{N - \alpha} = 1.
\end{equation*}
\end{proposition}

\begin{proof}%
\resetclaim%
The proof of Proposition~\ref{pufii} requires some preliminaries, stated in the form of claims.

\begin{claim}
\label{claimTheoremNCvgGst}
One has
\[
 \lim_{\alpha \to N} \Tilde{c}_\alpha^\gst = \kappa_{p, 1}.
\]
\end{claim}
We recall that the quantity \(\kappa_{p,1}\) is the groundstate level of the limiting problem \eqref{limit1} defined in \eqref{eqKappa}.

\begin{proofclaim}
Given $u \in H^1(\Rset^N)\setminus \{0\}$, by Lemma~\ref{lemmaConv} with $\xi_n=0$ we have, as \(\alpha \to N\),
\begin{equation*}
\begin{split}
  \Tilde{c}_{\alpha}^{\gst}
  &\leq \max_{t>0} J_{\alpha}( t u)
  = \frac{t^2}{2} \Big (\int_{\Rset^N} |\nabla u|^2 + \abs{u}^2 \Big) - \frac{t^{2p}}{2p} \int_{\Rset^N} (\Tilde{I}_\alpha \ast \abs{u}^p)\abs{u}^p  \\
  &\to \max_{t>0} \frac{t^2}{2} \Big (\int_{\Rset^N} |\nabla u|^2 + \abs{u}^2 \Big )- \frac{t^{2p}}{2p} \Big ( \int_{\Rset^N} \abs{u}^{p} \Big )^2 = \max_{t>0} \Psi_{p,1}(tu).
\end{split}
\end{equation*}
Taking the infimum with respect to \(u \in  H^1(\Rset^N)\setminus \{0\}\), we get that $\limsup_{\alpha \to N} \Tilde{c}_{\alpha}^{\gst} \leq \kappa_{p, 1}$. In particular, the groundstate solution $\Tilde{U}_\alpha$ of \eqref{cov} remains bounded in $H^1(\Rset^N)$ as $\alpha \to N$, since
\[
 \Tilde{c}_\alpha^{\gst} = \Big (\frac 1 2 - \frac{1}{2p} \Big) \Big (\int_{\Rset^N} |\nabla \Tilde{U}_\alpha|^2 + \abs{\Tilde{U}_\alpha}^2 \Big).
\]
Lemma~\ref{lemmaConvEstimate} yields
$$
  \liminf_{\alpha \to N} \Tilde{c}_{\alpha}^{\gst} =  \liminf_{\alpha \to N} \max_{t>0} J_{\alpha}(t U_\alpha) \geq   \liminf_{\alpha \to N} \max_{t>0} \Psi_{p,1}(tU_\alpha) \geq \kappa_{p,1},
$$
from which the reversed inequality follows.
\end{proofclaim}

\begin{claim}
\label{claimTheoremNBounded}
The family \((\Tilde{u}_\alpha)_{\alpha \in (\underline{\alpha}, N)}\) is bounded in \(H^1 (\Rset^N)\) as \(\alpha \to N\).

%\[
%  \limsup_{\alpha \to N} \int_{\Rset^N}\abs{\nabla \Tilde{u}_\alpha}^2 + \abs{\Tilde{u}_\alpha}^2
% <\infty.
%\]

\end{claim}
\begin{proofclaim}
We observe that, by Proposition~\ref{propositionControlGroundNodal},
\[
 \Bigl(\frac{1}{2} - \frac{1}{2 p}\Bigr)
 \int_{\Rset^N} \abs{\nabla \Tilde{u}_\alpha}^2 + \abs{\Tilde{u}_\alpha}^2
 = \Tilde{J}_\alpha (\Tilde{u}_\alpha)
 \le 2 \Tilde{c}_\alpha^\gst,
\]
and the conclusion follows then from Claim~\ref{claimTheoremNCvgGst}.
\end{proofclaim}

\begin{claim} \label{claimTheoremNSignChanges}
One has
\begin{equation*}
 \liminf_{\alpha \to N} \int_{\Rset^N} \abs{\nabla \Tilde{u}_\alpha^\pm}^2 + \abs{\Tilde{u}_\alpha^\pm}^2>0 \qquad \text{and} \qquad
 \liminf_{\alpha \to N} \int_{\Rset^N} \bigl(\Tilde{I}_\alpha \ast \abs{\Tilde{u}_\alpha^{\pm}}^p\bigr) \abs{\Tilde{u}_\alpha^\pm}^p > 0.
\end{equation*}
\end{claim}

\begin{proofclaim}
We recall that by the optimal Hardy--Littlewood--Sobolev inequality
\citelist{%
  \cite{LiebLoss2001}*{Theorem 4.3}%
  \cite{Lieb1983}*{Theorem 3.1}%
} %
for all functions \(f, g \in L^\frac{2N}{N + \alpha} (\Rset^N)\), we have
\begin{equation}
\label{ineqOptimalHLSPower}
  \int_{\Rset^N} (\Tilde{I}_\alpha \ast f) g
  \le \Tilde{C}_{N, \alpha} \Bigl(\int_{\Rset^N} \abs{f}^\frac{2N}{N + \alpha} \Bigr)^\frac{N + \alpha}{2 N}
   \Bigl(\int_{\Rset^N} \abs{g}^\frac{2N}{N + \alpha} \Bigr)^\frac{N + \alpha}{2 N},
\end{equation}
with an optimal constant \(\Tilde{C}_{N, \alpha}\) that can be expressed as
\begin{equation}
\label{ineqOptimalHLSPowerConstant}
  \Tilde{C}_{N, \alpha}
  = \frac{\pi^\frac{N - \alpha}{2} \Gamma (\frac{\alpha}{2})}
          {\Gamma (\frac{N + \alpha}{2}) }
  \biggl(\frac{\Gamma (N)}{\Gamma(\frac{N}{2})}\biggr)^\frac{\alpha}{N}.
\end{equation}
By the Hardy--Littlewood--Sobolev inequality \eqref{ineqOptimalHLSPower} and
by the Sobolev inequality, we observe that
\[
\begin{split}
 \int_{\Rset^N} \abs{\nabla \Tilde{u}_\alpha^\pm}^2 + \abs{\Tilde{u}_\alpha^\pm}^2
 &= \int_{\Rset^N} \bigl(\Tilde{I}_\alpha \ast \abs{\Tilde{u}_\alpha}^p\bigr) \abs{\Tilde{u}_\alpha^\pm}^p\\
 &\le C \Tilde{C}_{N, \alpha} \Bigl(\int_{\Rset^N} \abs{\nabla \Tilde{u}_\alpha}^2 + \abs{\Tilde{u}_\alpha}^2 \Bigr)^\frac{p}{2}
 \Bigl(\int_{\Rset^N} \abs{\nabla \Tilde{u}_\alpha^\pm}^2 + \abs{\Tilde{u}_\alpha^\pm}^2 \Bigr)^\frac{p}{2}.
\end{split}
\]
so that, since \(p > 2\), we have
\[
 1 \le C \Tilde{C}_{N, \alpha}\Bigl(\int_{\Rset^N} \abs{\nabla \Tilde{u}_\alpha}^2 + \abs{\Tilde{u}_\alpha}^2 \Bigr)^\frac{p}{2}
 \Bigl(\int_{\Rset^N} \abs{\nabla \Tilde{u}_\alpha^\pm}^2 + \abs{\Tilde{u}_\alpha^\pm}^2 \Bigr)^\frac{p - 2}{2}.
\]
In view of \eqref{ineqOptimalHLSPowerConstant}, we have
\[
 \lim_{\alpha \to N} \Tilde{C}_{N, \alpha}
  = 1,
\]
so that, by Claim~\ref{claimTheoremNBounded},
\begin{equation}
 \liminf_{\alpha \to N}
 \int_{\Rset^N}\abs{\nabla \Tilde{u}_\alpha^\pm}^2 + \abs{\Tilde{u}_\alpha^\pm}^2 = \liminf_{\alpha \to N} \int_{\Rset^N} \bigl(\Tilde{I}_\alpha \ast \abs{\Tilde{u}_\alpha}^p\bigr) \abs{\Tilde{u}_\alpha^\pm}^p > 0
\end{equation}

For the second estimate, we write, by the positive definiteness of the Riesz potential energy and the Cauchy--Schwarz inequality (see \cite{LiebLoss2001}*{Theorem 9.8}),
\[
\begin{split}
 \int_{\Rset^N} &\bigl(\Tilde{I}_\alpha \ast \abs{\Tilde{u}_\alpha^+}^p\bigr) \abs{\Tilde{u}_\alpha^+}^p\\
 &= \int_{\Rset^N} \bigl(\Tilde{I}_\alpha \ast \abs{\Tilde{u}_\alpha}^p\bigr) \abs{u_\alpha^+}^p
  - \int_{\Rset^N} \bigl(\Tilde{I}_\alpha \ast \abs{\Tilde{u}_\alpha^-}^p\bigr) \abs{\Tilde{u}_\alpha^+}^p\\
 &\ge  \int_{\Rset^N} \bigl(\Tilde{I}_\alpha \ast \abs{\Tilde{u}_\alpha}^p\bigr) \abs{u_\alpha^+}^p
 - \Big ( \int_{\Rset^N} \bigl(\Tilde{I}_\alpha \ast \abs{\Tilde{u}_\alpha^+}^p\bigr) \abs{\Tilde{u}_\alpha^+}^p \int_{\Rset^N} \bigl(\Tilde{I}_\alpha \ast \abs{\Tilde{u}_\alpha^-}^p\bigr) \abs{\Tilde{u}_\alpha^-}^p \Big )^\frac{1}{2}.
\end{split}
\]
The conclusion follows then from the fact that \(u_\alpha^+ \ne 0\),from the boundedness of the family \(u_\alpha^-\) in \(H^1 (\Rset^N)\) and from the Hardy--Littlewood--Sobolev inequality \eqref{ineqOptimalHLSPower}.
\end{proofclaim}

\begin{claim}
\label{cnodlimit}
We have
\[
 \lim_{\alpha \to N} \Tilde{c}_{\alpha}^{\nod}
 = 2 \kappa_{p, 2}
 = 2^\frac{p - 2}{p - 1} \kappa_{p, 1}.
\]
Moreover, define \(t_\alpha, s_\alpha \in (0, \infty)\) such that
\begin{equation}
 t_\alpha \Tilde{u}_\alpha^+ \in \mathcal{N}_{p,2}, \qquad \text{and} \qquad s_\alpha \Tilde{u}_\alpha^- \in \mathcal{N}_{p,2}
\end{equation}
where \(\mathcal{N}_{p,2}\) is the Nehari manifold associated to the functional \(\Psi_{p, 2}\) (see \eqref{othernehari}). Then, $t_\alpha$, $s_\alpha$ are bounded and the satisfy the following asymptotics as \(\alpha \to 0\):
\begin{gather}
\label{eqRieszYoungbis}
 t_\alpha^p s_\alpha^p\int_{\Rset^N} \bigl(I_\alpha \ast \abs{\Tilde{u}_\alpha^+}^p\bigr) \abs{\Tilde{u}_{\alpha}^-}^p = \frac{t_\alpha^{2 p}}{2}\int_{\Rset^N} \bigl(I_\alpha \ast \abs{\Tilde{u}_\alpha^+}^p\bigr) \abs{\Tilde{u}_{\alpha}^+}^p +
  \frac{s_\alpha^{2 p}}{2} \int_{\Rset^N} \bigl(I_\alpha \ast \abs{\Tilde{u}_\alpha^-}^p\bigr) \abs{\Tilde{u}_{\alpha}^-}^p + o(1),\\
\label{ciao2bis}
\Psi_{p, 2} (t_\alpha \Tilde{u}_\alpha^+) \to \kappa_{p, 2}, \ \ \Psi_{p, 2} (s_\alpha \Tilde{u}_\alpha^-) \to \kappa_{p, 2}.
\end{gather}
\end{claim}

\begin{proofclaim}
We take a function \(v \in C^\infty_c (\Rset^N)\) and we choose a vector \(\xi \in \Rset^N\) such that \(\abs{\xi} > \operatorname{diam}(\operatorname{supp} v )\).
We now define the function \(u : \Rset^N \to \Rset\) for each $x \in \R^N$ by \(u (x) = v (x) - v (x - \xi)\).
In view of Lemma~\ref{lemmaConv}, we have
\[
 \int_{\Rset^N} \bigl(\Tilde{I}_\alpha \ast \abs{t u^+ + s u^-}^p\bigr)\,\abs{t u^+ + s u^-}^p
 = \Bigl(\int_{\Rset^N} \abs{t u^+ + s u^-}^p\Bigr)^2 + o (1)\, (\abs{t}^{2 p} + \abs{s}^{2 p}).
\]
It follows therefore that
\begin{equation}
\label{eqlimTildeJPsi}
 \lim_{\alpha \to N} \max\, \bigl\{ \Tilde{J}_\alpha (t u^+ + s u^-) \st t, s \in [0, \infty)\}
 \le \max\, \{ \Psi_{p, 1} (t u^+ + s u^-) \st t, s \in [0, \infty) \bigr\}.
\end{equation}
Moreover, we have  for every \(s, t \in [0, \infty)\),
\begin{equation}
\begin{split}
\label{eqPsi12}
 \Psi_{p, 1} (t u^+ + s u^-)
 &= \frac{t^2 + s^2}{2} \int_{\Rset^N} \abs{\nabla v}^2 + \abs{v}^2 - \frac{(t^{p} + s^{p})^2}{2 p}
 \Bigl(\int_{\Rset^N} \abs{v}^p \Bigr)^2\\
 &\le r^2 \int_{\Rset^N} \abs{\nabla v}^2 + \abs{v}^2 - \frac{2 r^{2p}}{p}
 \Bigl(\int_{\Rset^N} \abs{v}^p \Bigr)^2 = 2 \Psi_{p, 2} (r v),
\end{split}
\end{equation}
where \(r = \sqrt{\frac{t^2 + s^2}{2}}\). By combining \eqref{eqlimTildeJPsi} and \eqref{eqPsi12}, we get, in view of the definition of the level \(\Tilde{c}_{\alpha}^\nod\)
\[
 \limsup_{\alpha \to N} \Tilde{c}_\alpha^\nod
 \le 2 \Psi_{p, 2} (r v).
\]
Since the latter inequality holds for every \(v \in C^\infty_c (\Rset^N)\) and since the set \(C^\infty_c (\Rset^N)\) is dense in the Sobolev space \(H^1 (\Rset^N)\), we have in view of the characterization \eqref{eqCharactKappa} and of the identity \eqref{eqKappa},
\[
 \limsup_{\alpha \to N} \Tilde{c}_\alpha^\nod \le 2 \kappa_{p, 2} = 2^\frac{p - 2}{p - 1} \kappa_{p, 1}.
\]

For the reversed inequality, first observe that
\[
t_\alpha^{2(p-1)}= \tfrac{\displaystyle\int_{\RN} |\nabla \Tilde{u}_\alpha^+|^2 + |\Tilde{u}_\alpha^+|^2 }{\displaystyle 2 \Big (\int_{\RN}  |\Tilde{u}_\alpha^+|^p \Big)^2 }
\qquad \text{and}\qquad
s_\alpha^{2(p-1)}= \tfrac{\displaystyle\int_{\RN} |\nabla \Tilde{u}_\alpha^-|^2 + |\Tilde{u}_\alpha^-|^2 }{\displaystyle2 \Big (\int_{\RN}  |\Tilde{u}_\alpha^-|^p \Big)^2 }.
\]
Combining Lemma~\ref{lemmaConvEstimate} and Claim~\ref{claimTheoremNSignChanges}, we conclude that $t_\alpha$ and $s_\alpha$ remain bounded and bounded away from $0$ as \(\alpha \to N\). Then,
\begin{equation}
\label{eqFirsurysr}
 \Tilde{c}_\alpha^\nod
 = \Tilde{J}_\alpha (\Tilde{u}_\alpha)
 \ge \Tilde{J}_\alpha (t_\alpha \Tilde{u}_\alpha^+ + s_\alpha  \Tilde{u}_\alpha^-).
\end{equation}
By the positive definiteness of the Riesz potential energy and by the Cauchy--Schwarz inequality (see \cite{LiebLoss2001}*{Theorem 9.8}), we have
\begin{equation}
\label{eqRieszYoung}
\begin{split}
 t_\alpha ^p s_\alpha ^p\int_{\Rset^N} \bigl(I_\alpha \ast \abs{\Tilde{u}_\alpha^+}^p\bigr) \abs{\Tilde{u}_{\alpha}^-}^p \le  \frac{t_\alpha^{2 p}}{2}\int_{\Rset^N} \bigl(I_\alpha \ast \abs{\Tilde{u}_\alpha^+}^p\bigr) \abs{\Tilde{u}_{\alpha}^+}^p +
  \frac{s_\alpha^{2 p}}{2} \int_{\Rset^N} \bigl(I_\alpha \ast \abs{\Tilde{u}_\alpha^-}^p\bigr) \abs{\Tilde{u}_{\alpha}^-}^p.
\end{split}
\end{equation}
Therefore, in view of \eqref{eqFirsurysr}, we deduce that
\begin{equation}
\label{eqTildecalphaDecoupled}
\begin{split}
 \Tilde{c}_\alpha^\nod
 \ge {} \frac{t_\alpha^2}{2} \int_{\Rset^N} \abs{\nabla \Tilde{u}_\alpha^+}^2 &+ \abs{\Tilde{u}_\alpha}^2
 -\frac{t_\alpha^{2 p}}{p} \int_{\Rset^N} (\Tilde{I}_\alpha \ast \abs{\Tilde{u}_\alpha^+}^p) \abs{\Tilde{u}_\alpha^+}^p\\
 &+ \frac{s_\alpha^2}{2} \int_{\Rset^N} \abs{\nabla \Tilde{u}_\alpha^-}^2 + \abs{\Tilde{u}_\alpha^-}^2
 -\frac{s_\alpha^{2p}}{p} \int_{\Rset^N} (\Tilde{I}_\alpha \ast \abs{\Tilde{u}_\alpha^-}^p) \abs{\Tilde{u}_\alpha^-}^p.
\end{split}
\end{equation}

By Lemma~\ref{lemmaConvEstimate}, we have, as \(\alpha \to N\),
\begin{equation} \label{ciao}
 \int_{\Rset^N} \bigl(\Tilde{I}_\alpha \ast \abs{\Tilde{u}_\alpha^\pm}^p\bigr) \abs{\Tilde{u}_\alpha^\pm}^p
 \le \Bigl(\int_{\Rset^N} \abs{\Tilde{u}_\alpha^\pm}^p\Bigr)^2 + O (N - \alpha).
\end{equation}
In view of \eqref{eqTildecalphaDecoupled} this leads us to
\begin{equation} \label{ciao2}
\Tilde{c}_\alpha^\nod
  \ge \Psi_{p, 2} (t_\alpha \Tilde{u}_\alpha^+) + \Psi_{p, 2} (s_\alpha \Tilde{u}_\alpha^-)
  + O (N - \alpha).
\end{equation}
By the characterization \eqref{eqCharactKappa} and by the identity \eqref{eqKappa}, it follows that
\[
 \liminf_{\alpha \to N} \Tilde{c}_\alpha^\nod
 \ge 2 \kappa_{p, 2} = 2^\frac{p - 2}{p - 1} \kappa_{p, 1},
\]
which proves the first part of the claim.

As a byproduct, the inequalities \eqref{eqRieszYoung} and \eqref{ciao2} become equalities in the limit $\alpha \to N$; this gives \eqref{eqRieszYoungbis} and \eqref{ciao2bis}.
\end{proofclaim}

\bigskip We are now in conditions to prove Proposition~\ref{pufii}. First, let us show that $t_{\alpha} \to 1$ and $s_{\alpha} \to 1$ as \(\alpha \to N\).
In view of  Claim~\ref{claimTheoremNBounded} and Lemma~\ref{lemmaConvEstimate}, and by using the positive definiteness of the Riesz potential energy and the Cauchy--Schwarz inequality (see \cite{LiebLoss2001}*{Theorem 9.8}) we have: \begin{equation*}
\begin{split}
\Tilde{c}_\alpha^\nod = \Tilde{J}_\alpha (\Tilde{u}_\alpha^+ + \Tilde{u}_\alpha^-)
 &\ge \frac{1}{2} \int_{\Rset^N} \abs{\nabla u_\alpha^+}^2 + \abs{u_\alpha^+}^2
  - \frac{1}{p}\int_{\Rset^N} \bigl( I_\alpha \ast \abs{u_\alpha^+}^p\bigr)\abs{u_\alpha^+}^p\\
  &\qquad
   +\frac{1}{2} \int_{\Rset^N} \abs{\nabla u_\alpha^-}^2 + \abs{u_\alpha^-}^2
  - \frac{1}{p}\int_{\Rset^N} \bigl( I_\alpha \ast \abs{u_\alpha^-}^p\bigr)\abs{u_\alpha^-}^p\\
 &\ge \Psi_{p, 2} (\Tilde{u}_\alpha^+) + \Psi_{p, 2} (\Tilde{u}_\alpha^-) + o (1).
\end{split}
\end{equation*}
By Claim~\ref{cnodlimit}, we have\begin{equation*}
\begin{split}
 2 \kappa_{p, 2} &\ge \Psi_{p, 2} (t_\alpha \Tilde{u}_\alpha^+) + \Psi_{p, 2} (s_\alpha \Tilde{u}_\alpha^-)\\
 &\qquad+ \Psi_{p, 2} (\Tilde{u}_\alpha^+) - \Psi_{p, 2} (t_\alpha \Tilde{u}_\alpha^+) + \Psi_{p, 2} (\Tilde{u}_\alpha^-) - \Psi_{p, 2} (s_\alpha \Tilde{u}_\alpha^-) + o (1)\\
 &\ge  2 \kappa_{p, 2} +
 \frac{1}{2}\biggl(1 - \Bigl(1 - \frac{1}{p}\Bigr)t_\alpha^2 - \frac{1}{p t_\alpha^{2 p - 2}} \biggr)
 \int_{\Rset^N} \abs{\nabla \Tilde{u}_\alpha^+}^2 + \abs{\Tilde{u}_\alpha^+}^2 \\
 &\qquad+
 \frac{1}{2}\biggl(1 - \Bigl(1 - \frac{1}{p}\Bigr)s_\alpha^2 - \frac{1}{p s_\alpha^{2 p - 2}} \biggr)
 \int_{\Rset^N} \abs{\nabla \Tilde{u}_\alpha^-}^2 + \abs{\Tilde{u}_\alpha^-}^2 + o (1)
\end{split}
\end{equation*}
Since the integrals on the right-hand side remain bounded away from \(0\) (Claim~\ref{claimTheoremNSignChanges}), we have
\begin{align*}
 \lim_{n \to \infty} 1 - \Bigl(1 - \frac{1}{p}\Bigr)t_\alpha^2 - \frac{1}{p t_\alpha^{2 p - 2}} &= 0,&
 &\text{and}&
 \lim_{n \to \infty} 1 - \Bigl(1 - \frac{1}{p}\Bigr)s_\alpha^2 - \frac{1}{p s_\alpha^{2 p - 2}} & = 0.
\end{align*}
By Young's inequality, we have for each \(\tau \in (0, \infty)\),
\[
  1 \le \frac{1}{p} \frac{1}{\tau^{2 p - 2}} + \Bigl(1 - \frac{1}{p}\Bigr) \tau^2,
\]
Therefore the function \( \theta : (0, \infty) \to \Rset\) defined for every \(\tau \in (0, \infty)\)
\[
 \theta (\tau) = 1 - \Bigl(1 - \frac{1}{p}\Bigr)\tau^2 - \frac{1}{p \tau^{2 p - 2}},
\]
is nonnegative and \(\theta (\tau) = 0\) if and only if \(\tau = 1\).
Since we have \(\lim_{\tau \to 0} \theta (\tau) = \infty\) and \(\lim_{\tau \to \infty} \theta (\tau) = \infty\), we conclude that $t_\alpha \to 1$ and  $s_\alpha \to 1$ as \(\alpha \to N\).

By \eqref{ciao2bis}, the families $t_\alpha u_\alpha^+$, $s_\alpha u_\alpha^-$ minimize the functional $\Psi_{p,2}$ restricted to its Nehari manifold $\mathcal{N}_{p,2}$ (as \(\alpha \to N\) ). Lemma~\ref{salut} implies the existence of vectors $\Tilde{\xi}_\alpha^+$, $\Tilde{\xi}_\alpha^- \in \R^N$ such that
\begin{equation} \label{solounavez}
  u_{\alpha}^+ - V(\cdot - \Tilde{\xi}_\alpha^+) \to 0,\quad
  u_{\alpha}^- - V(\cdot - \Tilde{\xi}_\alpha^-) \to 0 \quad \text{in \(H^1 (\Rset^N)\)},
\end{equation}
as \(\alpha \to N\)
where $V=V_2$ is the groundstate of problem \eqref{limit1} for $\mu=2$. In particular,

\[
  u_{\alpha} - (V(\cdot - \Tilde{\xi}_\alpha^+) - V(\cdot - \Tilde{\xi}_\alpha^-) ) \to 0 \quad \text{in \(H^1 (\Rset^N)\)}.
\]

If the sequence $\Tilde{\xi}_\alpha^+ - \Tilde{\xi}_\alpha^-$ were bounded, taking positive and negative part of the above expression yields a contradiction with \eqref{solounavez}. Hence $|\Tilde{\xi}_\alpha^+ - \Tilde{\xi}_\alpha^-| \to +\infty$.
Then,
\[
 \lim_{\alpha \to N} \ \ \inf_{\xi^+, \xi^- \in \Rset^N} \norm{u_{\alpha} -  (V (\cdot - \xi^+) - V (\cdot - \xi^-))}_{H^1 (\Rset^N)} = 0.
\]
Since by Fatou's lemma,
\[ \liminf_{\abs{\xi^+} + \abs{\xi^-} \to +\infty} \norm{\Tilde{u}_\alpha - (V (\cdot - \xi^+) - V (\cdot - \xi^-))}_{H^1 (\Rset^N)} \ge \min\,\bigl\{\norm{\Tilde{u}_\alpha}_{H^1 (\Rset^N)}, \norm{V}_{H^1(\Rset^N)}\bigr\},
\]
the function
\[
  (\xi^+, \xi^-) \in \Rset^N \times \Rset^N \longmapsto \norm{\Tilde{u}_\alpha - (V (\cdot - \xi^+) - V (\cdot - \xi^-))}_{H^1 (\Rset^N)}
\]
attains its infimum at some pair of vectors \((\xi_{\alpha}^+, \xi_{\alpha}^-) \in \Rset^N \times \Rset^N\) for $\alpha$ sufficiently close to \(N\). As in Section \ref{sectionProofTheoremMain0} we can conclude that $| \xi_\alpha^{\pm} - \tilde{\xi}_\alpha^\pm| \to 0$; in particular, $|{\xi}_\alpha^+ - {\xi}_\alpha^-| \to +\infty$.

\medskip Finally, we prove that \(\abs{\xi_\alpha^+ - \xi_\alpha^-}^{N-\alpha} \to 1\) as \(\alpha \to N\). By \eqref{eqRieszYoungbis},
$$
  \int_{\Rset^N}  \bigl(I_\alpha \ast \abs{\Tilde{u}_\alpha^+}^p\bigr) \abs{\Tilde{u}_{\alpha}^-}^p \to \frac{1}{2}\int_{\Rset^N} \bigl(I_\alpha \ast \abs{\Tilde{u}_\alpha^+}^p\bigr) \abs{\Tilde{u}_{\alpha}^+}^p +
  \frac{1}{2} \int_{\Rset^N} \bigl(I_\alpha \ast \abs{\Tilde{u}_\alpha^-}^p\bigr) \abs{\Tilde{u}_{\alpha}^-}^p.
$$
But,
\begin{align*} \int_{\Rset^N}  \bigl(I_\alpha \ast \abs{\Tilde{u}_\alpha^+}^p\bigr) \abs{\Tilde{u}_{\alpha}^-}^p = \int_{\Rset^N}  (I_\alpha \ast V(\cdot- \xi_\alpha^+)^p) V(\cdot- \xi_\alpha^-)^p + o(1) \\ =\int_{\Rset^N}  (I_\alpha \ast V(\cdot- (\xi_\alpha^+ - \xi_\alpha^-)^p) V^p + o(1),\end{align*}
where in the last equality we have just made a change of variables. Ana6logously,
\[
  \int_{\Rset^N} \bigl(I_\alpha \ast \abs{\Tilde{u}_\alpha^\pm}^p\bigr) \abs{\Tilde{u}_{\alpha}^\pm}^p
  = \int_{\Rset^N}  \bigl(I_\alpha \ast V(\cdot- \xi_\alpha^\pm)^p\bigr)\, V(\cdot- \xi_\alpha^\pm)^p + o(1)
  = \int_{\Rset^N}  \bigl(I_\alpha \ast V^p\bigr)\, V^p + o(1).
\]
Lemma~\ref{lemmaConv} implies that $\abs{\xi_\alpha^+ - \xi_\alpha^-}^{N - \alpha} \to 1$, concluding the proof.
\end{proof}

\resetstep

\begin{proof}[Proof of Theorem~\ref{theoremMainN}]
With Proposition~\ref{pufii} in hand, we follow the same ideas used to prove Theorem~\ref{theoremMain0}. Also here we can assume without loss of generality that \(\xi_{\alpha}^+ = 0\) and \(\xi_{\alpha}^- = \xi_{\alpha}= (m_\alpha, 0, \dotsc, 0)\).

By Proposition~\ref{pufii}, we have:
\begin{align}
\label{key_estimate2}
  \lim_{\alpha \to N} m_\alpha = \lim_{\alpha \to N} \abs{\xi_\alpha}  &= \infty,&
  &\text{ and }&
  \lim_{\alpha \to N} m_\alpha^{N-\alpha} = \lim_{\alpha \to N} \abs{\xi_{\alpha}}^{N-\alpha} &= 1.
\end{align}
Again we define then \(R_\alpha\) to be the orthogonal reflection of \(\Rset^N\) that sends \(0\) to \(\xi_\alpha\),
that is for each \(x  \in \Rset^N\)
\[
 R_{\alpha} (x) = (m_\alpha -x_1, x_2, \dotsc, x_N).
\]
We also define the functions \(\Breve{u}_\alpha = \Tilde{u}_\alpha \circ R_\alpha\), and \(v_\alpha = \Tilde{u}_\alpha + \Breve{u}_\alpha\), and the half-space
\[
 \Omega_\alpha = \bigl\{ x \in \Rset^N \st \xi_\alpha \cdot x  < \abs{\xi_\alpha}^2/2\bigr\}.
\]
By construction, \(v_\alpha \circ R_\alpha = v_\alpha\), and
thus the function \(v_\alpha\) is even with respect to \(\partial \Omega_\alpha\).

By Proposition~\ref{pufii}, we have \(v_\alpha \to 0\) in \(H^1(\Rset^N)\) as \(\alpha \to N\).
We will show that for \(\alpha\) sufficiently close to \(N\), we have \(v_\alpha=0\), that is, the solution \(\Tilde{u}_\alpha\) has an odd reflection symmetry with respect to the hyperplane \(\partial \Omega_\alpha\).

\begin{step}
\label{claimTheoremNOrthogonality}
For every direction \(\zeta \in \Rset^N\),
\begin{align*}
  \scalprod{v_\alpha}{\partial_\zeta V}_{H^1 (\Rset^N)} & = 0,&
  &\text{ and }&
  \scalprod{v_\alpha}{\partial_\zeta V}_{H^1 (\Omega_\alpha)} & = o \bigl(\norm{v_\alpha}_{H^1 (\Omega_\alpha)}\big).
\end{align*}
\end{step}

The proof is the same as Step 1 in Theorem~\ref{theoremMain0}.

\begin{step}
The function \(v_\alpha\) satisfies the linear equation
\begin{equation} \label{extra}
 \mathcal{L}_\alpha v_\alpha = 0 \qquad \text{in \(\Rset^N\)},
\end{equation}
where the linear differential operator \(\mathcal{L}_\alpha\) is defined by
\[
  \mathcal{L}_\alpha v = -\Delta v + v - (\Tilde{I}_\alpha \ast (G_\alpha v)) H_\alpha
 - (\Tilde{I}_\alpha \ast K_\alpha ) L_\alpha v,
\]
with
\begin{align*}
 G_\alpha
 &=  \begin{cases}
   \frac{\abs{\Tilde{u}_\alpha}^p - \abs{\Breve{u}_\alpha}^p}{\Tilde{u}_\alpha + \Breve{u}_\alpha}& \text{where \(\Tilde{u}_\alpha \ne - \Breve{u}_\alpha\)},\\
   p \abs{\Tilde{u}_\alpha}^{p - 2} \Tilde{u}_\alpha & \text{elsewhere},
  \end{cases}\\
 H_\alpha
 &=\frac{1}{2} \bigl(\abs{\Tilde{u}_\alpha}^{p - 2} \Tilde{u}_\alpha
 - \abs{\Breve{u}_\alpha}^{p - 2} \Breve{u}_\alpha\bigr),\\
  K_\alpha &=\frac{1}{2} (\abs{\Tilde{u}_\alpha}^p + \abs{\Breve{u}_\alpha}^p),\\
  L_\alpha &= \begin{cases}
   \frac{\abs{\Tilde{u}_\alpha}^{p - 2} \Tilde{u}_\alpha
 + \abs{\Breve{u}_\alpha}^{p - 2} \Breve{u}_\alpha}{\Tilde{u}_\alpha + \Breve{u}_\alpha} & \text{where \(\Tilde{u}_\alpha \ne -\Breve{u}_\alpha\)},\\
   (p - 1) \abs{\Tilde{u}_\alpha}^{p - 2} & \text{elsewhere}.
  \end{cases}
\end{align*}
\end{step}

Again, the proof is identical to that of Step 2 of Theorem~\ref{theoremMain0}.

\medskip

\begin{step} Conclusion of the proof of Theorem~\ref{theoremMainN}.
\end{step}

The idea here is also very closely related to that of Theorem~\ref{theoremMain0}; the main difference is in the way one passes to the limit. As commented above, Theorem~\ref{theoremMainN} follows if we show that $v_{\alpha}=0$ for $\alpha$ sufficiently close $N$. Let us assume by contradiction that there is a sequence \((\alpha_n)_{n \in \Nset}\)
in \((\underline{\alpha}, N)\), $\alpha_n \to N$ such that \(v_{\alpha_n} \ne 0\). We define for each \(n \in \Nset\) the normalized functions
\[
  w_n = \frac{v_{\alpha_n}}{\norm{v_{\alpha_n}}}.
\]
Without loss of generality, we can assume that the sequence \((w_n)_{n \in \Nset}\) converges weakly in \(H^1(\Rset^N)\) to some function \(w \in H^1 (\Rset^N)\). \medskip

By Proposition~\ref{pufii}, we have that
\begin{equation}
\label{eqTheoremNCvgcs}
\begin{aligned}
 G_\alpha- p\, \bigl(V^{p-1} - V^{p-1}(\cdot - \xi_\alpha)\bigr) & \to 0
    & &\text{in } L^q(\Rset^N),
    & & \text{if } \tfrac{1}{2}-\tfrac{1}{N} \le \tfrac{p - 1}{q} \le \tfrac{1}{2}\\
 H_\alpha - \bigl(V^{p-1} - V^{p-1}(\cdot - \xi_\alpha)\bigr)  & \to 0
    & &\text{in } L^q(\Rset^N),
    & & \text{if } \tfrac{1}{2}-\tfrac{1}{N} \le \tfrac{p - 1}{q} \le \tfrac{1}{2},\\
 K_\alpha - \bigl( V^p + V^p(\cdot- \xi_\alpha)\bigr) & \to 0
     & &\text{in } L^q(\Rset^N),
    & & \text{if } \tfrac{1}{2}-\tfrac{1}{N} \le \tfrac{p}{q} \le \tfrac{1}{2},\\
 L_\alpha - (p-1) \bigl( V^{p-2} + V^{p-2}(\cdot- \xi_\alpha)\bigr)& \to 0
    & & \text{in } L^q(\Rset^N),
    & & \text{if } \tfrac{1}{2}-\tfrac{1}{N} \le \tfrac{p}{q} \le \tfrac{1}{2}.
\end{aligned}
\end{equation}

We test the equation \eqref{extra} against \(\varphi \in C_0^{\infty}(\Rset^N)\) and we obtain, in view of \eqref{eqTheoremNCvgcs},
\begin{multline}
\label{eqylan}
\int_{\Rset^N} (\nabla w_n \cdot \nabla \varphi + w_n \varphi) \\
 = p \int_{\Rset^N} \int_{\Rset^N} \frac{ \bigl(V^{p-1}(x) - V^{p-1}(x - \xi_{\alpha_n})\bigr)\, w_n(x)\, \bigl(V^{p-1}(y) - V^{p-1}(y - \xi_{\alpha_n})\bigr)\, \varphi(y)}{\abs{x - y}^{N-{\alpha_n}}}    \dif x \dif y \\  + (p-1) \int_{\Rset^N} \int_{\Rset^N} \frac{  \bigl(V^{p}(x) + V^{p}(x - \xi_{\alpha_n})\bigr)\, \bigl(V^{p-2}(y) +V^{p-2}(y - \xi_{\alpha_n})\bigr)\,w_n(y)\, \varphi(y)}{\abs{x - y}^{N-{\alpha_n}}}   \dif x \dif y\\ + o(1).
\end{multline}

We claim that
\begin{equation}
\label{last}
\int_{\Rset^N} \int_{\Rset^N} \frac{ \bigl(V^{p-1}(x) - V^{p-1}(x - \xi_{\alpha_n})\bigr)\, w_n(x)\, \bigl(V^{p-1}(y) - V^{p-1}(y - \xi_{\alpha_n})\bigr)\, \varphi(y)}{\abs{x - y}^{N-{\alpha_n}}}    \dif x \dif y \to 0.
\end{equation}
Indeed, we observe that $V^{p-1} w_n \to V^{p-1} w$  and $(V^{p-1}(y) - V^{p-1}(y - \xi_{\alpha_n})) \varphi(y) \to V^{p-1}(y) \varphi(y)$ in $L^q(\Rset^N)$, for $q \in [1, \frac{2N}{N-2})$. By Lemma \ref{lemmaConv},
\begin{align*} \lim_{n \to +\infty} \int_{\Rset^N} \int_{\Rset^N} \frac{ V^{p-1}(x) w_n(x) (V^{p-1}(y) - V^{p-1}(y - \xi_{\alpha_n})) \varphi(y)}{\abs{x - y}^{N-{\alpha_n}}}    \dif x \dif y \\
= \Bigl(\int_{\Rset^N}  V^{p-1} w_n\Bigr)  \Bigl(\int_{\Rset^N}
V^{p-1} \varphi \Bigr). \end{align*}

Moreover, by the evenness of \(w_n\) with respect to \(\partial \Omega_\alpha\), and since \(R_\alpha (\xi_\alpha) = 0\), we have \(\abs{R_\alpha (z) - \xi_{\alpha}} = \abs{z}\). Recalling that \(V\) is radially symmetric, we have by changes of variable \(\check{x}=R_\alpha(x)\) and \(\check{y}= R_\alpha (y)\),
\begin{multline*}
 \int_{\Rset^N} \int_{\Rset^N} \frac{ V^{p-1}(x-\xi_{\alpha_n})\, w_n(x)\, \bigl(V^{p-1}(y) - V^{p-1}(y - \xi_{\alpha_n})\bigr)\, \varphi(y)}{\abs{x - y}^{N-{\alpha_n}}}    \dif x \dif y\\
 =\int_{\Rset^N} \int_{\Rset^N} \frac{ V^{p-1}(\check{x})\, w_n(\check{x})\,\bigl(V^{p-1}(\check{y} - \xi_{\alpha_n}) - V^{p-1}(\check{y})\bigr)\, \check{\varphi}(\check{y} - \xi_{\alpha_n})}{\abs{\check{x} - \check{y}}^{N-{\alpha_n}}}    \dif \check{x} \dif \check{y},
\end{multline*}
where \(\check{\varphi} (y_1, y_2, \dotsc, y_N) = \varphi (-y_1, y_2, \dotsc, y_N)\).
Again by Lemma \ref{lemmaConv} and by the radial symmetry of \(V\),
\begin{align*}
\lim_{n \to +\infty} \int_{\Rset^N} \int_{\Rset^N} \frac{ V^{p-1}(x-\xi_{\alpha_n})\, w_n(x) \,(V^{p-1}(y) - V^{p-1}(y - \xi_{\alpha_n})) \varphi(y)}{\abs{x - y}^{N-{\alpha_n}}}    \dif x \dif y \\ = \Bigl(\int_{\Rset^N}  V^{p-1} w_n\Bigr) \Bigl( \int_{\Rset^N}
V^{p-1} \check{\varphi} \Bigr)
=\Bigl(\int_{\Rset^N}  V^{p-1} w_n\Bigr) \Bigl( \int_{\Rset^N}
V^{p-1} \varphi \Bigr),
\end{align*}
Hence \eqref{last} follows.

\medskip Reasoning analogously and recalling that $\varphi$ has compact support, the second term in the right-hand side of \eqref{eqylan} converges to
\[
  2 (p-1) \Bigl(\int_{\Rset^N}  V^{p}\Bigr)  \Bigl(\int_{\Rset^N}V^{p-2} w  \varphi\Bigr).
\]
We conclude that $w$ is a (weak) solution of
\[
 -\Delta w + w = 2 (p - 1) \norm{V}_{L^p (\Rset^N)}^p V^{p - 2} w.
\]
By Step~\ref{claimTheoremNOrthogonality} and the nondegeneracy of \eqref{limit0} (recall that $V=V_2$ is a groundstate solution of \eqref{limit1} for $\mu=2$), we have \(w = 0\).

We now multiply the equation \eqref{extra} by the function $v_{{\alpha_n}}$, integrate and divide by $\norm{v_{{\alpha_n}}}_{H^1 (\Rset^N)}^2$, to obtain:
\begin{equation*}
\begin{split}
 1 =&  \int_{\Rset^N} \abs{\nabla w_n}^2 + \abs{w_n}^2 = \int_{\Rset^N} \bigl(I_{\alpha_n} \ast (G_{\alpha_n} w_n)\bigr) H_{\alpha_n} w_n
 + (I_{\alpha_n} \ast K_{\alpha_n}) L_{\alpha_n} w_n^2 + o(1) \\
 = &\,   p \int_{\Rset^N} \int_{\Rset^N} \frac{  \bigl(V^{p-1}(x) - V^{p-1}(x - \xi_{\alpha_n})\bigr)\, w_n(x)\, \bigl(V^{p-1}(y) - V^{p-1}(y - \xi_{\alpha_n})\bigr)\, w_n (y)}{\abs{x - y}^{N-{\alpha_n}}} \dif x \dif y \\
 & + (p-1) \int_{\Rset^N} \int_{\Rset^N} \frac{\bigl(V^{p}(x) + V^{p}(x - \xi_{\alpha_n})\bigr)\, \bigl(V^{p-2}(y) +V^{p-2}(y - \xi_{\alpha_n})\bigr)\,w^2_{n}(y) }{\abs{x - y}^{N-{\alpha_n}}}    \dif x \dif y\\ & \qquad+ o(1).
\end{split}
\end{equation*}
We argue again as in the proof of \eqref{last} to conclude that the first term in the right-hand side converges to $0$.  Again, Lemma~\ref{lemmaConv} and \eqref{key_estimate2} imply that the second term in the right-hand side converges to
\[
  4 (p-1)  \Bigl(\int_{\Rset^N}  V^{p} \Bigr)\Bigl(\int_{\Rset^N} V^{p-2} w^2 \Bigr) = 0,
\]
since $V^{p-2} w_n^2 \to V^{p-2} w^2=0$ strongly in $L^q(\Rset^N)$, for $q \in [1, \frac{N}{N-2})$. This yields the desired contradiction and concludes the proof of Theorem~\ref{theoremMainN}.
\end{proof}


\begin{bibdiv}

\begin{biblist}

\bib{bahri-lions}{article}{
    author = {Bahri, Abbas },
    author={Lions, Pierre-Louis },
     title = {On the existence of a positive solution of semilinear elliptic
              equations in unbounded domains},
   journal = {Ann. Inst. H. Poincar\'e Anal. Non Lin\'eaire},
    volume = {14},
      year = {1997},
    number = {3},
     pages = {365--413},
%       DOI = {10.1016/S0294-1449(97)80142-4},
}

\bib{BenciCerami1987}{article}{
   author={Benci, Vieri},
   author={Cerami, Giovanna},
   title={Positive solutions of some nonlinear elliptic problems in exterior
   domains},
   journal={Arch. Rational Mech. Anal.},
   volume={99},
   date={1987},
   number={4},
   pages={283--300},
   issn={0003-9527},
%    review={\MR{898712}},
%    doi={10.1007/BF00282048},
}

%\bib{Bogachev2007}{book}{
%   author={Bogachev, V. I.},
%   title={Measure theory},
%   publisher={Springer},
%   place={Berlin},
%   date={2007},
%   isbn={978-3-540-34513-8},
%   isbn={3-540-34513-2},
%   doi={10.1007/978-3-540-34514-5},
%}

\bib{CastroCossioNeuberger1997}{article}{
   author={Castro, Alfonso},
   author={Cossio, Jorge},
   author={Neuberger, John M.},
   title={A sign-changing solution for a superlinear Dirichlet problem},
   journal={Rocky Mountain J. Math.},
   volume={27},
   date={1997},
   number={4},
   pages={1041--1053},
   issn={0035-7596},
}

\bib{CastorCossioNeuberger1998}{article}{
   author={Castro, Alfonso},
   author={Cossio, Jorge},
   author={Neuberger, John M.},
   title={A minmax principle, index of the critical point, and existence of
   sign-changing solutions to elliptic boundary value problems},
   journal={Electron. J. Differential Equations},
   volume={1998},
   date={1998},
   number={2},
   pages={18},
   issn={1072-6691},
}

\bib{CeramiSoliminiStruwe1986}{article}{
   author={Cerami, G.},
   author={Solimini, S.},
   author={Struwe, M.},
   title={Some existence results for superlinear elliptic boundary value
   problems involving critical exponents},
   journal={J. Funct. Anal.},
   volume={69},
   date={1986},
   number={3},
   pages={289--306},
   issn={0022-1236},
}

    \bib{CingolaniClappSecchi2012}{article}{
      author={Cingolani, Silvia},
      author={Clapp, M{\'o}nica},
      author={Secchi, Simone},
      title={Multiple solutions to a magnetic nonlinear Choquard equation},
      journal={Z. Angew. Math. Phys.},
      volume={63},
      date={2012},
      number={2},
      pages={233--248},
      issn={0044-2275},
      %doi={10.1007/s00033-011-0166-8},
    }

    \bib{CingolaniSecchi}{article}{
   author={Cingolani, Silvia},
   author={Secchi, Simone},
   title={Multiple \(\mathbb{S}^{1}\)-orbits for the Schr\"odinger-Newton system},
   journal={Differential and Integral Equations},
   volume={26},
   number={9/10},
   pages={867--884},
   date={2013},
   %eprint={arXiv:1304.7639},
}

\bib{ClappSalazar}{article}{
   author={Clapp, M{\'o}nica},
   author={Salazar, Dora},
   title={Positive and sign changing solutions to a nonlinear Choquard
   equation},
   journal={J. Math. Anal. Appl.},
   volume={407},
   date={2013},
   number={1},
   pages={1--15},
   issn={0022-247X},
%   doi={10.1016/j.jmaa.2013.04.081},
}

\bib{Diosi1984}{article}{
   title={Gravitation and quantum-mechanical localization of macro-objects},
   author={Di\'osi, L.},
   journal={Phys. Lett. A},
   volume={105},
   number={4--5},
   date={1984},
   pages={199--202},
}

\bib{GhimentiMorozVanSchaftingen}{article}{
  author = {Ghimenti, Marco},
  author = {Moroz, Vitaly},
  author= {Van Schaftingen, Jean},
  title = {Least action nodal solutions for the quadratic Choquard equation},
  journal={to appear in Proc. Amer. Math. Soc.}
  eprint = {arXiv:1511.04779},
}

\bib{GhimentiVanSchaftingen}{article}{
  author = {Ghimenti, Marco},
  author = {Van Schaftingen, Jean},
  title = {Nodal solutions for the Choquard equation},
  journal={J. Funct. Anal.},
  volume={271},
  date={2016},
  number={1},
  pages={107--135},
}

\bib{Jones1995}{article}{
  title={Newtonian quantum gravity},
  author={Jones, K. R. W.},
  journal={Austral. J. Phys.},
  volume={48},
  number={6},
  pages={1055--1081},
  year={1995},
}

\bib{Kwong1989}{article}{
   author={Kwong, Man Kam},
   title={Uniqueness of positive solutions of $\Delta u-u+u^p=0$ in ${\bf
   R}^n$},
   journal={Arch. Rational Mech. Anal.},
   volume={105},
   date={1989},
   number={3},
   pages={243--266},
   issn={0003-9527},
   %DOI={10.1007/BF00251502},
}

\bib{Lieb1977}{article}{
   author={Lieb, Elliott H.},
   title={Existence and uniqueness of the minimizing solution of Choquard's
   nonlinear equation},
   journal={Studies in Appl. Math.},
   volume={57},
   date={1976/77},
   number={2},
   pages={93--105},
}

\bib{Lieb1983}{article}{
   author={Lieb, Elliott H.},
   title={Sharp constants in the Hardy-Littlewood-Sobolev and related
   inequalities},
   journal={Ann. of Math. (2)},
   volume={118},
   date={1983},
   number={2},
   pages={349--374},
   issn={0003-486X},
%    review={\MR{717827 (86i:42010)}},
%    doi={10.2307/2007032},
}

\bib{LiebLoss2001}{book}{
   author={Lieb, Elliott H.},
   author={Loss, Michael},
   title={Analysis},
   series={Graduate Studies in Mathematics},
   volume={14},
   edition={2},
   publisher={American Mathematical Society},
   place={Providence, RI},
   date={2001},
   pages={xxii+346},
   isbn={0-8218-2783-9},
}

\bib{Lions1980}{article}{
   author={Lions, P.-L.},
   title={The Choquard equation and related questions},
   journal={Nonlinear Anal.},
   volume={4},
   date={1980},
   number={6},
   pages={1063--1072},
   issn={0362-546X},
}

\bib{Menzala1980}{article}{
   author={Menzala, Gustavo Perla},
   title={On regular solutions of a nonlinear equation of Choquard's type},
   journal={Proc. Roy. Soc. Edinburgh Sect. A},
   volume={86},
   date={1980},
   number={3-4},
   pages={291--301},
   issn={0308-2105},
}

\bib{Moroz-Penrose-Tod-1998}{article}{
   author={Moroz, Irene M.},
   author={Penrose, Roger},
   author={Tod, Paul},
   title={Spherically-symmetric solutions of the Schr\"odinger-Newton
   equations},
   journal={Classical Quantum Gravity},
   volume={15},
   date={1998},
   number={9},
   pages={2733--2742},
   issn={0264-9381},
}

\bib{MVSGNLSNE}{article}{
   author={Moroz, Vitaly},
   author={Van Schaftingen, Jean},
   title={Groundstates of nonlinear Choquard equations: Existence,
   qualitative properties and decay asymptotics},
   journal={J. Funct. Anal.},
   volume={265},
   date={2013},
   number={2},
   pages={153--184},
   issn={0022-1236},
%   doi={10.1016/j.jfa.2013.04.007},
}

\bib{MVSReview}{article}
{
   author={Moroz, Vitaly},
   author={Van Schaftingen, Jean},
   title={A guide to the Choquard equation},
   eprint={arXiv:1606.02158},
}

\bib{Oh1990}{article}{
   author={Oh, Yong-Geun},
   title={On positive multi-lump bound states of nonlinear Schr\"odinger
   equations under multiple well potential},
   journal={Comm. Math. Phys.},
   volume={131},
   date={1990},
   number={2},
   pages={223--253},
   issn={0010-3616},
}

\bib{Pekar1954}{book}{
  author = {Pekar, S.I.},
  title = {Untersuchungen \"uber die Elektronentheorie der Kristalle},
  publisher = {Akademie-Verlag},
  address={Berlin},
  year = {1954},
  pages = {29-34},
}

\bib{Penrose1996}{article}{
  author = {Penrose, Roger},
  title = {On gravity's role in quantum state reduction},
  journal = {Gen. Rel. Grav.},
  volume = {28},
  year = {1996},
  number = {5},
  pages = {581--600},
}

\bib{Tod-Moroz-1999}{article}{
   author={Tod, Paul},
   author={Moroz, Irene M.},
   title={An analytical approach to the Schr\"odinger-Newton equations},
   journal={Nonlinearity},
   volume={12},
   date={1999},
   number={2},
   pages={201--216},
   issn={0951-7715},
}

\bib{Weinstein1985}{article}{
   author={Weinstein, Michael I.},
   title={Modulational stability of groundstates of nonlinear Schr\"odinger
   equations},
   journal={SIAM J. Math. Anal.},
   volume={16},
   date={1985},
   number={3},
   pages={472--491},
   issn={0036-1410},
%    review={\MR{783974 (86i:35130)}},
   %DOI={10.1137/0516034},
}

    \bib{Weth2001}{thesis}{
      author={Weth, Tobias},
      title={Spectral and variational characterizations
	of solutions to semilinear eigenvalue problems},
      date={2001},
      organization={Johannes Gutenberg-Universit\"at},
      address={Mainz},
    }

%\bib{Weth2006}{article}{
%   author={Weth, Tobias},
%   title={Energy bounds for entire nodal solutions of autonomous superlinear
%   equations},
%   journal={Calc. Var. Partial Differential Equations},
%   volume={27},
%   date={2006},
%   number={4},
%   pages={421--437},
%   issn={0944-2669},
%   %review={\MR{2263672}},
%   %doi={10.1007/s00526-006-0015-3},
%}

\bib{Willem1996}{book}{
   author={Willem, Michel},
   title={Minimax theorems},
   series={Progress in Nonlinear Differential Equations and their
   Applications, 24},
   publisher={Birkh\"auser},
   address={Boston, Mass.},
   date={1996},
   pages={x+162},
   isbn={0-8176-3913-6},
%    review={\MR{1400007}},
%    doi={10.1007/978-1-4612-4146-1},
}

%\bib{Willem2013}{book}{
%  author = {Willem, Michel},
%  title = {Functional analysis},
%  subtitle = {Fundamentals and Applications},
%  series={Cornerstones},
%  publisher = {Birkh\"auser},
%  place = {Basel},
%  volume = {XIV},
%  pages = {213},
%  date={2013},
%}

\end{biblist}

\end{bibdiv}
\end{document}